\documentclass{amsart}

\usepackage{amssymb,amsmath,stmaryrd,mathrsfs,amsthm}

\usepackage[all,2cell]{xy}
\UseAllTwocells

\usepackage{tikz}
\tikzset{lab/.style={auto,font=\scriptsize}} 
\usetikzlibrary{arrows}

\usepackage{enumitem}
\usepackage{xcolor}
\definecolor{darkgreen}{rgb}{0,0.45,0} 
\usepackage[pagebackref,colorlinks,citecolor=darkgreen,linkcolor=darkgreen]{hyperref}
\usepackage{mathtools}
\usepackage{url}                
\usepackage{xspace}             

\newcommand{\D}{\ensuremath{\mathscr{D}}\xspace}
\newcommand{\cA}{\ensuremath{\mathcal{A}}\xspace}
\newcommand{\cB}{\ensuremath{\mathcal{B}}\xspace}
\newcommand{\cC}{\ensuremath{\mathcal{C}}\xspace}
\newcommand{\cD}{\ensuremath{\mathcal{D}}\xspace}
\newcommand{\bC}{\ensuremath{\mathcal{C}}\xspace}
\newcommand{\bM}{\ensuremath{\mathcal{M}}\xspace}
\newcommand{\bW}{\ensuremath{\mathcal{W}}\xspace}

\usepackage{mathbbol}
\newcommand{\bbone}{\ensuremath{\mathbb{1}}\xspace}
\newcommand{\bbtwo}{\ensuremath{\mathbb{2}}\xspace}
\newcommand{\bbthree}{\ensuremath{\mathbb{3}}\xspace}

\newcommand{\op}{^{\mathrm{op}}}

\newcommand{\coop}{^{\mathrm{coop}}}
\def\ho{\mathscr{H}\!\mathit{o}\xspace}
\newcommand{\fib}{\mathsf{fib}}
\newcommand{\cof}{\mathsf{cof}}
\def\shift#1#2{{#1}^{#2}}

\newcommand{\ot}{\ensuremath{\leftarrow}}
\let\toto\rightrightarrows
\let\into\hookrightarrow
\let\xto\xrightarrow

\def\toiso{\xto{\smash{\raisebox{-.5mm}{$\scriptstyle\sim$}}}}

\def\cCat{\ensuremath{\mathcal{C}\mathit{at}}\xspace}
\def\cCAT{\ensuremath{\mathcal{C}\mathit{AT}}\xspace}

\def\DT#1#2#3#4#5#6#7{%
  \xymatrix{{#1} \ar[r]^-{#2} &
    {#3} \ar[r]^-{#4} &
    {#5} \ar[r]^-{#6} &
    {#7}
  }}
\def\DTl#1#2#3#4#5#6#7{%
  \xymatrix@C=3pc{{#1} \ar[r]^-{#2} &
    {#3} \ar[r]^-{#4} &
    {#5} \ar[r]^-{#6} &
    {#7}
  }}

\makeatletter
\def\defthm#1#2{%
  \newtheorem{#1}{#2}[section]%
  \expandafter\def\csname #1autorefname\endcsname{#2}%
  \expandafter\let\csname c@#1\endcsname\c@thm}
\makeatother
\theoremstyle{plain}
\newtheorem{thm}{Theorem}[section]

\defthm{cor}{Corollary}
\defthm{lem}{Lemma}
\theoremstyle{definition}
\defthm{defn}{Definition}
\theoremstyle{remark}
\defthm{rmk}{Remark}
\defthm{eg}{Example}
\defthm{egs}{Examples}
\defthm{warn}{Warning}

\setitemize[1]{leftmargin=2em}
\setenumerate[1]{leftmargin=*}

\makeatletter
\let\c@equation\c@thm
\numberwithin{equation}{section}
\makeatother

\mathtoolsset{showonlyrefs,showmanualtags}

\title{Mayer-Vietoris sequences in stable derivators}

\author{Moritz Groth}
\email{M.Groth@math.ru.nl}
\address{Department of Mathematics,
  Radboud University,
  Heyendaalseweg 135, 
  6525 AJ Nijmegen,
  Netherlands}
\thanks{The first author was partially supported by the Deutsche Forschungsgemeinschaft within the graduate program `Homotopy and Cohomology' (GRK 1150) and by the Dutch Science Foundation (NWO).}

\author{Kate Ponto}
\email{kate.ponto@uky.edu}
\address{Department of Mathematics,
  University of Kentucky,
  719 Patterson Office Tower,
  Lexington, KY, 40506,
  USA}
\thanks{The second author was partially supported by NSF grant DMS-1207670.}

\author{Michael Shulman}
\email{shulman@sandiego.edu}
\address{Department of Mathematics and Computer Science,
  University of San Diego,
  5998 Alcal\'a Park,
  San Diego, CA, 92117,
  USA}
\thanks{The third author was partially supported by an NSF postdoctoral fellowship and NSF grant DMS-1128155, and appreciates the hospitality of the University of Kentucky.
Any opinions, findings, and conclusions or recommendations expressed in this material are those of the authors and do not necessarily reflect the views of the National Science Foundation.}
\date{\today}

\begin{document}

\begin{abstract}
  We show that stable derivators, like stable model categories, admit Mayer-Vietoris sequences arising from cocartesian squares.
  Along the way we characterize homotopy exact squares, and give a detection result for colimiting diagrams in derivators.
  As an application, we show that a derivator is stable if and only if its suspension functor is an equivalence.
\end{abstract}

\maketitle

\tableofcontents

\section{Introduction}

There are many axiomatizations of ``stable homotopy theory'', including stable model categories~\cite{hovey:modelcats}, 
stable $(\infty,1)$-categories~\cite{lurie:ha,lurie:higher-topoi,joyal}, and triangulated categories~\cite{neeman:triang, verdier:thesis}.
In this paper we study a less well-known member of this family, the \emph{stable derivators}.
This notion was defined (under different names) by Grothendieck~\cite{grothendieck:derivateurs}, Franke~\cite{franke:triang}, and 
Heller~\cite{heller:htpythys}, and has been studied more recently in \cite{m:introder, cisinski:idcm, groth:ptstab}.
It is closely related to the other notions mentioned above: stable model categories and stable $(\infty,1)$-categories both give rise to stable derivators, while each (strong) stable derivator has an underlying triangulated category.

Among these choices,  derivators are a convenient level of generality because they are better-behaved than triangulated categories, while requiring less technical machinery than model categories or $(\infty,1)$-categories.
To be precise, a stable derivator enhances a derived category or homotopy category (such as a triangulated category) by also including the homotopy categories of diagrams of various shapes.
This allows homotopy limits and colimits to be characterized by ordinary universal properties.

The main concrete result of this paper is that \emph{Mayer-Vietoris sequences} exist in any stable derivator.
Precisely, this means that for any cocartesian square
\begin{equation}
  \vcenter{\xymatrix{
      x\ar[r]^f\ar[d]_g & y\ar[d]^j\\
      z\ar[r]_k & w,
    }}
\end{equation}
there is
a distinguished triangle (i.e.\ a fiber and cofiber sequence)
\begin{equation}
  \DTl{x}{(f,-g)}{y \oplus z}{[j,k]}{w}{}{\Sigma x.}\label{eq:intro-mvseq}
\end{equation}
As in any triangulated category, this yields long exact sequences in homology and cohomology.
The corresponding result for stable model categories is~\cite[Lemma~5.7]{may:traces}, while for triangulated categories there is no notion of ``cocartesian square'' other than the existence of~\eqref{eq:intro-mvseq}.
It seems that such a result for stable $(\infty,1)$-categories is not yet explicitly in the literature; we can now conclude it from our theorem about stable derivators.

However, the point of our new proof is not so much that it applies to new examples, but that it advances the theory of stable derivators, which can be easier to work with for the above-mentioned reasons.
Indeed, along the way we improve some basic computational lemmas about derivators, such as a characterization of homotopy exact squares and a tool to identify colimiting subdiagrams.
We conclude by showing that a derivator is stable (i.e.\ cocartesian and cartesian squares agree) if and only if it is cofiber-stable (i.e.\ its cofiber and fiber functors are equivalences) if and only if it is $\Sigma$-stable (i.e.\ its suspension and loop space functors are equivalences).

We begin in \S\ref{sec:derivators} by establishing notation and recalling basic definitions and facts about derivators.
In \S\ref{sec:hoexactsq} and \S\ref{sec:kan-cart} we give the aforementioned characterizations and detection results.
Then in \S\ref{sec:pointedstable} we recall definitions and facts about pointed and stable derivators, and in \S\ref{sec:mv} we establish our main result on the existence of Mayer-Vietoris sequences.
Finally, in \S\ref{sec:characterization} we prove the equivalence of the three notions of stability.
In Appendix~\ref{sec:mates} we summarize the ``calculus of mates'' for natural transformations, which is used extensively in the theory of derivators.

This paper can be thought of as a sequel to~\cite{groth:ptstab} and a prequel to \cite{gps:additivity} and \cite{gs:enriched}.

\subsection*{Acknowledgments}

We would like to thank the referee for suggesting a more elegant proof of \autoref{prop:pushout-dt}.

\section{Review of derivators}
\label{sec:derivators}

In this section we recall the definition of a derivator and fix some notation and terminology.
Let \cCat and \cCAT denote the 2-categories of small and large categories, respectively. We write $\bbone$ for the terminal category.   
For an object $a\in A$ we also use $a$ to denote the functor $\bbone\to A$ whose value on the object of $\bbone$ is $a$.

\begin{defn}
  A \textbf{derivator} is a 2-functor $\D\colon\cCat\op\to\cCAT$ with the following properties.
  \begin{itemize}[leftmargin=4em]
  \item[(Der1)] $\D\colon \cCat\op\to\cCAT$ takes coproducts to products.  In particular, $\D(\emptyset)$ is the terminal category.
  \item[(Der2)] For any $A\in\cCat$, the family of functors $a^*\coloneqq \D(a)\colon \D(A) \to\D(\bbone)$, as $a$ ranges over the objects of $A$, is jointly conservative (isomorphism-reflecting).
  \item[(Der3)] Each functor $u^*\coloneqq \D(u)\colon \D(B) \to\D(A)$ has both a left adjoint $u_!$ and a right adjoint $u_*$.
  \item[(Der4)] For any functors $u\colon A\to C$ and $v\colon B\to C$ in \cCat, let $(u/v)$ denote their comma category, with projections $p\colon(u/v) \to A$ and $q\colon(u/v)\to B$.
    If $B=\bbone$ is the terminal category, then the canonical mate-transformation
    (see \autoref{sec:mates})
    \[ q_! p^* \to q_! p^* u^* u_! \to q_! q^* v^* u_! \to v^* u_! \]
    is an isomorphism.
    Similarly, if $A=\bbone$ is the terminal category, then the canonical mate-transformation
    \[ u^* v_* \to p_* p^* u^* v_* \to p_* q^* v^* v_* \to p_* q^* \]
    is an isomorphism.
  \end{itemize}
\end{defn}

There are also notions of morphisms of derivators and transformations between such morphisms giving rise to a well-behaved 2-category of derivators, but we will have no need of them in this paper.

\begin{warn}
  Ours is one of two possible conventions for the definition of a derivator;
  it is based on the idea that the basic example should consist of (covariant) \emph{diagrams} (see \autoref{eq:derivator}).
  The other convention, which defines a derivator to be a 2-functor $\cCat\coop \to \cCAT$ (where $\cCat\coop$ denotes reversal of both 1-cells and 2-cells), requires that the basic example consist of \emph{presheaves} (i.e.\ contravariant diagrams).
  Our convention is that of Heller~\cite{heller:htpythys} and Franke~\cite{franke:triang}; the other convention was used by Grothendieck~\cite{grothendieck:derivateurs} and Cisinski~\cite{cisinski:idcm}.
  The two definitions are equivalent, by composition with the isomorphism $(-)\op : \cCat\op \to \cCat\coop$, but the directions of various 2-cells in each convention are reversed with respect to the other.
\end{warn}

For a functor $u\colon A\to B$, we write $u^*\colon\D(B) \to \D(A)$ for its image under the 2-functor \D, and refer to 
it as \textbf{restriction} along $u$. The category $\D(\bbone)$ is the \textbf{underlying category} of \D. We call the 
objects of $\D(A)$ \textbf{(coherent, $A$-shaped) diagrams} in \D, motivated by the following examples.

\begin{eg}\label{eq:derivator}
  Any (possibly large) category \bC gives rise to a \textbf{represented} 2-functor~$y(\bC)$ defined by
  \[ y(\bC)(A) \coloneqq \bC^A. \]
  Its underlying category is \bC itself.  
  If \bC is both complete and cocomplete, then $y(\bC)$ is a derivator (and in fact the converse also holds).
  The functors $u_!$ and $u_*$ are left and right Kan extensions respectively, and when $B$ is the terminal category they compute colimits and limits.
  Axiom (Der4) expresses the fact that Kan extensions can be computed ``pointwise'' from limits and colimits, as in~\cite[X.3.1]{maclane}.
\end{eg}

Any coherent diagram $X\in\D(A)$ has an \textbf{underlying (incoherent) diagram}, which is an ordinary diagram in $\D(\bbone)$, i.e.\ an object of the functor category $\D(\bbone)^A$.
For each $a\in A$, the underlying diagram of $X$ sends $a$ to $a^*X$. We may also write $X_a$ for $a^*X$. More generally, any $X\in\D(B\times A)$ has an underlying ``partially coherent'' diagram which is an object of $\D(B)^A$, sending $a\in A$ to the diagram $X_a = a^*X = (1_B\times a)^*X$.

We will occasionally refer to a coherent diagram as having the \textbf{form} of, or \textbf{looking like}, its underlying diagram, and proceed to draw that underlying diagram using objects and arrows in the usual way.
It is very important to note, though, that in a general derivator, a coherent diagram is not determined by its underlying diagram, not even up to isomorphism.
This is the case in the following two examples, which are the ones of primary interest.

\begin{eg}\label{eg:model}
  Suppose \bC is a \emph{Quillen model category} (see e.g.~\cite{hovey:modelcats}), with class \bW of \emph{weak equivalences}.
  Its \emph{derived} or \textbf{homotopy derivator} $\ho(\bC)$ is defined from $y(\bC)$ by formally inverting the pointwise weak equivalences:
  \[ \ho(\bC)(A) \coloneqq (\bC^A)[(\bW^A)^{-1}].\]
  See~\cite{cisinski:idcm,groth:ptstab} for proofs that this defines a derivator.
  Its underlying category is the usual homotopy category $\bC[\bW^{-1}]$ of \bC, while its functors $u_!$ and $u_*$ are the left and right derived functors, respectively, of those for $y(\bC)$.
\end{eg}

\begin{eg}
  If $\cC$ is a \emph{complete and cocomplete $(\infty,1)$-category} as in~\cite{lurie:higher-topoi,joyal}, then it has a \textbf{homotopy derivator} defined by
  \[ \ho(\cC)(A) \coloneqq \mathrm{Ho}(\cC^A) \]
  where $\mathrm{Ho}$ denotes the usual homotopy category of an $(\infty,1)$-category, obtained by identifying equivalent morphisms.
  Since this fact does not seem to appear in the literature, we briefly sketch a proof.

  Axiom (Der1) is easy, while (Der2) follows from~\cite[Theorem 5.C]{joyal}.
  For (Der3) and (Der4), let $\bC$ be a simplicial category which presents $\cC$, and
  let $\mathrm{SSET}$ denote the category of simplicial sets in a higher universe, so that there is a simplicial Yoneda embedding $\bC \to \mathrm{SSET}^{\bC\op}$.
  Now the projective model structure on $\mathrm{SSET}^{\bC\op}$ gives rise to a derivator $\ho(\mathrm{SSET}^{\bC\op})$ (in the higher universe).
  As with any model category, its right Kan extension functors are computed in terms of model-categorical homotopy limits.
  Moreover, by~\cite[5.1.1.1]{lurie:higher-topoi}, $\mathrm{SSET}^{\bC\op}$ presents the $(\infty,1)$-category $\infty\mathit{GPD}^{\cC\op}$, and these homotopy limits present $(\infty,1)$-categorical limits therein.

  Finally, the simplicial Yoneda embedding $\bC \to \mathrm{SSET}^{\bC\op}$ presents the $(\infty,1)$-cate\-gorical Yoneda embedding $\cC \to \infty\mathit{GPD}^{\cC\op}$, which by~\cite[5.1.3.2]{lurie:higher-topoi} is fully faithful and closed under small limits.
  Thus, each $\ho(\cC)(A)$ can be identified with a full subcategory of $\ho(\mathrm{SSET}^{\bC\op})(A) = \ho(\infty\mathit{GPD}^{\cC\op})(A)$.
  Since the right Kan extension functors of the latter are computed by small $(\infty,1)$-categorical limits, they preserve the image of $\ho(\cC)(A)$.
  Thus, $\ho(\cC)$ inherits the ``right half'' of (Der3) and (Der4) from $\ho(\mathrm{SSET}^{\bC\op})$.
  Applying the same argument to $\cC\op$ yields the other half of these axioms.

  Finally, we note that if \cC is presented by a combinatorial model category \bM, then the derivator $\ho(\cC)$ constructed above agrees with the derivator $\ho(\bM)$ constructed in \autoref{eg:model}.
  If $\bM = \mathbf{sSet}^B$, then this follows essentially from~\cite[5.1.1.1]{lurie:higher-topoi}.
  It is straightforward to verify that left Bousfield localization of a model category of simplicial presheaves corresponds to accessible localization of the $(\infty,1)$-category it presents, and both act pointwise on functor categories; thus we can extend the claim from simplicial presheaves to any Bousfield localization thereof.
  Finally, by~\cite{dug:pres}, any combinatorial model category is equivalent to a localization of a simplicial presheaf category.
\end{eg}

Following established terminology for~$(\infty,1)$-categories, in a general derivator we refer to the functors $u_!$ and $u_*$ in (Der3) as \textbf{left and right Kan extensions}, respectively, rather than \emph{homotopy} Kan extensions.
This is unambiguous since actual `categorical' Kan extensions are meaningless for an abstract derivator.
Similarly, when the target category $B$ is $\bbone$, we call them \textbf{colimits and limits}. In this language, (Der4) says that right and left Kan extensions are ``computed pointwise'' in terms of limits and colimits.

Note that (Der1) and (Der3) together imply that each category $\D(A)$ has (actual) small coproducts and products.

We say that a derivator is \textbf{strong} if it satisfies:
  \begin{itemize}[leftmargin=4em]
  \item[(Der5)] For any $A$, the induced functor $\D(A\times \bbtwo) \to \D(A)^\bbtwo$ is full and essentially surjective, where $\bbtwo=(0\to 1)$ is the category with two objects and one nonidentity arrow between them.
  \end{itemize}
Represented derivators and homotopy derivators associated to model categories or $\infty$-categories are strong.

\begin{rmk}\label{rmk:der5}
  Axiom (Der5) is necessary whenever we want to perform limit constructions starting with morphisms in the underlying category $\D(\bbone)$, since it enables us to ``lift'' such morphisms to objects of $\D(\bbtwo)$.
  Combined with (Der2), it implies that if two objects of $\D(A\times \bbtwo)$ become isomorphic in $\D(A)^\bbtwo$, then they were already isomorphic in $\D(A\times \bbtwo)$ --- although such an isomorphism is not in general uniquely determined by its image in $\D(A)^\bbtwo$.

  The exact form of axiom (Der5) is also negotiable; for instance, Heller~\cite{heller:htpythys} assumed a stronger version in which \bbtwo is replaced by any finite free category.
\end{rmk}

The following examples are also often useful.

\begin{eg}\label{eg:shifted}
  For any derivator \D and category $B\in\cCat$, we have a \textbf{shifted} derivator $\shift\D B$ defined by $\shift\D B(A) \coloneqq \D(B\times A)$.
  This is technically very convenient: it enables us to ignore extra ``parameter'' categories $B$ by shifting them into the (universally quantified) derivator under consideration.
  Note that $\shift{y(\bC)}{B} \cong y(\bC^B)$ and $\shift{\ho(\bC)}{B} \cong \ho(\bC^B)$.

  \label{eg:opposite}
  Similarly, the \textbf{opposite} derivator of \D is defined by $\D\op(A) \coloneqq \D(A\op)\op$.
  Note that $y(\bC)\op = y(\bC\op)$ and $\ho(\bC\op) = \ho(\bC)\op$ and $(\shift\D B)\op = \shift{(\D\op)}{B\op}$.

  If \D is strong, so are $\shift{\D}{B}$ and $\D\op$;
  see~\cite[Theorem~1.25]{groth:ptstab}.
\end{eg}

\section{Homotopy exact squares}
\label{sec:hoexactsq}

The primary tool for calculating with Kan extensions in derivators is the notion of a \emph{homotopy exact square} (see also \cite{maltsiniotis:exact}), which is defined as follows.
Suppose given any natural transformation in \cCat which lives in a square
\begin{equation}
  \vcenter{\xymatrix{
      D\ar[r]^p\ar[d]_q \drtwocell\omit{\alpha} &
      A\ar[d]^u\\
      B\ar[r]_v &
      C.
    }}\label{eq:hoexactsq'}
\end{equation}
Then by 2-functoriality of \D, we have an induced transformation
\[\vcenter{\xymatrix{
    \D(C) \ar[r]^{u^*}\ar[d]_{v^*}\drtwocell\omit{\alpha^*} &
    \D(A) \ar[d]^{p^*}\\
    \D(B) \ar[r]_{q^*} &
    \D(D).
  }}\]
As summarized in \autoref{sec:mates}, this transformation has mates
\begin{gather}
  q_! p^* \to q_! p^* u^* u_! \xto{\alpha^*} q_! q^* v^* v_! \to v^* u_!  \mathrlap{\qquad\text{and}}\label{eq:hoexmate1'}\\
  u^* v_* \to p_* p^* u^* v_* \xto{\alpha^*} p_* q^* v^* v_* \to p_* q^*,\label{eq:hoexmate2'}
\end{gather}
of which one is an isomorphism if and only if the other is.

\begin{defn}\label{defn:hoexact}
  A square~\eqref{eq:hoexactsq'} is \textbf{homotopy exact} if the two mate-transforma\-tions~\eqref{eq:hoexmate1'} and~\eqref{eq:hoexmate2'} are isomorphisms in any derivator \D.
\end{defn}

For instance, Axiom (Der4) asserts that the following canonical squares are homotopy exact:
\begin{equation}
  \vcenter{\xymatrix{
      (u/c)\ar[r]^-p\ar[d]_q \drtwocell\omit{\alpha} &
      A\ar[d]^u\\
      \bbone\ar[r]_c &
      C
    }}
  \qquad\text{and}\qquad
  \vcenter{\xymatrix{
      (c/v)\ar[r]^-p\ar[d]_q \drtwocell\omit{\alpha} &
      \bbone\ar[d]^c\\
      B \ar[r]_v &
      C
    }}
\end{equation}
Some other examples are:
\begin{itemize}
\item If~\eqref{eq:hoexactsq'} is a pullback square, then it is homotopy exact if $u$ is an opfibration or $v$ is a fibration; see~\cite[Prop.~1.24]{groth:ptstab}.
\item By the functoriality of mates, the horizontal or vertical pasting of homotopy exact squares is homotopy exact.
\item If $u\colon A\to C$ is fully faithful, then the identity $u\circ 1_A = u\circ 1_A$ is homotopy exact; see~\cite[Prop.~1.20]{groth:ptstab}.
  Thus, in this case $u_!$ and $u_*$ are fully faithful.
\item For any $u\colon A\to C$ and $v\colon B\to C$, the comma square
  \begin{equation}
  \vcenter{\xymatrix@-.5pc{
      (u/v)\ar[r]^-p\ar[d]_q \drtwocell\omit &
      A\ar[d]^u\\
      B\ar[r]_v &
      C
      }}
  \end{equation}
  is homotopy exact; see~\cite[Proposition~1.26]{groth:ptstab}.
  In other words, the conditions ``$B=\bbone$'' or ``$A=\bbone$'' in (Der4) can be removed.
\end{itemize}
The main result of this section is a characterization result for all homotopy exact squares, which reduces them to the following simpler notion.

\begin{defn}\label{def:hocontr}
  A small category $A$ is \textbf{homotopy contractible} if the counit
  \begin{equation}
    (\pi_A)_!(\pi_A)^* \to 1_{\D(\bbone)}\label{eq:hocontr'}
  \end{equation}
  is an isomorphism in any derivator.
\end{defn}
We use $\pi_A$ to denote the functor $A\xto{} \bbone$.  More generally, we will use $\pi$ to denote projections.

Thus, $A$ is homotopy contractible if the $A$-shaped diagram which is constant at an object $x\in\D(\bbone)$ has colimit $x$.
Note that in a represented derivator~$y(\bC)$,~\eqref{eq:hocontr'} is an isomorphism whenever $A$ is connected.
However, 
being homotopy contractible is a much stronger condition than being connected.

\begin{defn}
  For a square as in~\eqref{eq:hoexactsq'}, and $a\in A$, $b\in B$, and $\gamma\colon u(a)\to v(b)$, let
  \((a/D/b)_\gamma\)
  be the category of triples $(d\in D, a\xto{\phi} p(d), q(d) \xto{\psi}b)$ such that $v\psi \circ \alpha_d \circ u\phi = \gamma$.
\end{defn}

\begin{thm}\label{thm:hoexchar}
  If $(a/D/b)_\gamma$ is homotopy contractible for all $a$, $b$, and $\gamma$, then~\eqref{eq:hoexactsq'} is homotopy exact.
\end{thm}
\begin{proof}
By (Der2) and (Der4), homotopy exactness of~\eqref{eq:hoexactsq'} is equivalent to homotopy exactness of the pasted squares 
\begin{equation}
  \vcenter{\xymatrix@-.5pc{
      (q/b)\ar[r]\ar[d] \drtwocell\omit &
      D\ar[r]^p\ar[d]_q \drtwocell\omit{\alpha} &
      A\ar[d]^u\\
      \bbone\ar[r]_b &
      B\ar[r]_v &
      C
    }}
\end{equation}
for all $b\in B$.  Extending the diagram and applying the same argument, 
homotopy exactness of~\eqref{eq:hoexactsq'} is equivalent to all of the following pastings being homotopy exact, where the objects of $(a/D/b)$ are triples $(d\in D, a\xto{\phi} p(d), q(d) \xto{\psi}b)$.
\begin{equation}
  \vcenter{\xymatrix@-.5pc{
      (a/D/b) \ar[rr]\ar[d] \drrtwocell\omit &&
      \bbone\ar[d]^a\\
      (q/b)\ar[r]\ar[d] \drtwocell\omit &
      D\ar[r]^p\ar[d]_q \drtwocell\omit{\alpha} &
      A\ar[d]^u\\
      \bbone\ar[r]_b &
      B\ar[r]_v &
      C.
    }}\label{eq:hoexreduction}
\end{equation}
The top-right and left-bottom composites around the boundary of~\eqref{eq:hoexreduction} applied to $(d,\phi,\psi)\in (a/D/b)$ yield $ua$ and $vb$, respectively, and the corresponding component of the pasted natural transformation~\eqref{eq:hoexreduction} is the composite
\begin{equation}
  ua \xto{u\phi} upd \xto{\alpha_d} vqd \xto{v\psi} vb.
\end{equation}
However, the square
\begin{equation}
  \vcenter{\xymatrix@-.5pc{
      C(ua,vb)\ar[r]\ar[d] \drtwocell\omit &
      \bbone\ar[d]^{u a}\\
      \bbone\ar[r]_{v b} &
      C
    }}\label{eq:hoexcomma}
\end{equation}
is also homotopy exact, since the discrete category $C(ua,vb)$ is also the comma category $(ua/vb)$.
Now~\eqref{eq:hoexreduction} factors through~\eqref{eq:hoexcomma} by a functor $k_{a,b}\colon (a/D/b) \to C(ua,vb)$, whose fiber over $\gamma$ is $(a/D/b)_\gamma$.
Thus, by the functoriality of mates,~\eqref{eq:hoexreduction} is homotopy exact if and only if the induced transformation
\begin{equation}
  \pi_! (k_{a,b})_! (k_{a,b})^* \pi^* \to \pi_! \pi^*\label{eq:hoextrans}
\end{equation}
is an isomorphism, where $\pi\colon C(ua,vb)\to\bbone$.
Finally, by (Der1), $(k_{a,b})^*$ is the product of all the induced functors $\D(\bbone) \to \D((a/D/b)_\gamma)$, and likewise for $(k_{a,b})_!$.
Thus, if each $(a/D/b)_\gamma$ is homotopy contractible,~\eqref{eq:hoextrans} is an isomorphism.
\end{proof}

As a particular case of \autoref{thm:hoexchar}, we obtain a sufficient condition for a functor to be \emph{homotopy final}.  

\begin{defn}\label{defn:hofinal}
  A functor $f\colon A\to B$ is called \textbf{homotopy final} if the square
  \[\vcenter{\xymatrix@-.5pc{
      A\ar[r]^f\ar[d] &
      B\ar[d]\\
      \bbone\ar[r] &
      \bbone
    }}
  \]
  is homotopy exact, i.e.\ for any $X\in\D(B)$, the colimits of~$X$ and of $f^* X$ agree.
\end{defn}

In particular, if $f\colon A\to B$ is homotopy final, then $A$ is homotopy contractible if and only if $B$ is.

\begin{cor}\label{thm:hofinal}
  A functor $f\colon A\to B$ is homotopy final if for each $b\in B$, the comma category $(b/f)$ is homotopy contractible.
\end{cor}

Any right adjoint is homotopy final by~\cite[Prop.~1.18]{groth:ptstab}.
Thus, if two categories are connected by an adjunction, each is homotopy contractible if and only if the other is.
Often the easiest way to verify homotopy contractibility of a small category is to connect it to \bbone with a zigzag of adjunctions.
In particular, any category with an initial or terminal object is homotopy contractible.

\begin{rmk}
  Although we will not need it in this paper, the converse of \autoref{thm:hoexchar} also holds by the following argument.
  A functor $f\colon A\to B$ is a \textbf{homotopy equivalence} if the map
  \begin{equation}\label{eq:hoeqv}
    (\pi_A)_! (\pi_A)^* \cong (\pi_B)_! f_! f^* (\pi_B)^* \to (\pi_B)_! (\pi_B)^*
  \end{equation}
  is an isomorphism in any derivator.
  Heller~\cite{heller:htpythys} and Cisinski~\cite{cisinski:presheaves} showed that a functor $f$ is a homotopy equivalence if and only if  the nerve of $f$ is a weak homotopy equivalence.

  Now the proof of \autoref{thm:hoexchar} shows that~\eqref{eq:hoexactsq'} is homotopy exact if and only if each $k_{a,b}$ is a homotopy equivalence.
  Since $k_{a,b}$ is the disjoint union of the functors $k_{a,b,\gamma}\colon (a/D/b)_\gamma \to\bbone$, by Heller and Cisinski's characterization $k_{a,b}$ 
  is a homotopy equivalence if and only if each of $k_{a,b,\gamma}$ is.
  This implies that also the converse of \autoref{thm:hofinal} is true, i.e., a functor $f\colon A\to B$ is homotopy final if and only if the comma categories $(b/f)$ have weakly contractible nerves.

  We do not know how to show the converse of \autoref{thm:hoexchar} without using a characterization such as Heller and Cisinski's.
  It is true that if $f = \bigsqcup f_i : \bigsqcup A_i \to \bigsqcup B_i$ is a coproduct of functors, then the map~\eqref{eq:hoeqv} in any derivator splits up as a coproduct of the corresponding maps for the $f_i$'s.
  However, in a general derivator a coproduct of maps can be an isomorphism even if not all the summands are.
  For instance, in $\mathbf{Set}\op$, and in the category of commutative rings, the terminal object is annihilating for coproducts, so that the coproduct of any map with the identity of the terminal object is again the identity of the terminal object.

  It is true in a \emph{pointed} derivator (see \S\ref{sec:pointedstable}) that if a coproduct of maps is an isomorphism then each summand must be an isomorphism, since in the pointed case, every summand of a coproduct is naturally a retract of it.
  Moreover, as soon as~\eqref{eq:hoeqv} is an isomorphism in every pointed derivator, the functor $f$ must be a homotopy equivalence --- but the only proof we know of this latter fact uses Heller and Cisinski's characterization, noting that the morphism $\ho(\mathbf{sSet}) \to \ho(\mathbf{sSet_*})$ which adjoins a disjoint basepoint is cocontinuous and conservative.
  The analogous argument fails for a general derivator $\D$, since while it does has a ``pointed variant'' $\D_*$, the map $\D\to\D_*$ may not be conservative.
  For instance, if $\D$ is represented by $\mathbf{Set}\op$ or commutative rings, then $\D_*$ is trivial.
  (By contrast, asking that~\eqref{eq:hoeqv} is an isomorphism in every \emph{stable} derivator is a genuinely weaker statement, corresponding to the nerve of $f$ being a \emph{stable} homotopy equivalence.)
\end{rmk}

Heller and Cisinski's characterization also implies that the notions of homotopy exact square, homotopy equivalence functor, and homotopy contractible category are not actually dependent on the definition of a derivator.

\begin{thm}\label{thm:hoex}
  For a square~\eqref{eq:hoexactsq'} in \cCat, the following are equivalent.
  \begin{enumerate}
  \item The square is homotopy exact, i.e.\ the mate-transformation $q_! p^* \to v^* u_!$ is an isomorphism in any derivator \D.\label{item:he1}
  \item As in~\ref{item:he1}, but only for derivators of the form $\ho(\bC)$ for \bC a model category.\label{item:he2}
  \item As in~\ref{item:he1}, but only for derivators of the form $\ho(\cC)$ for \cC a complete and cocomplete $(\infty,1)$-category.\label{item:he4}
  \item As in~\ref{item:he1}, but only for the particular derivator $\ho(\mathbf{sSet}) = \ho(\infty\mathit{Gpd})$.\label{item:he3}
  \item Each functor $k_{a,b}$ is a homotopy equivalence.\label{item:he5}
  \item Each nerve $N k_{a,b}$ is a weak homotopy equivalence of simplicial sets.\label{item:he6}
  \item Each category $(a/D/b)_\gamma$ is homotopy contractible.\label{item:he7}
  \item Each nerve $N(a/D/b)_\gamma$ is a weakly contractible simplicial set.\label{item:he8}
  \end{enumerate}
\end{thm}

\section{Detection lemmas}
\label{sec:kan-cart}

In this section we discuss several lemmas for detecting when certain diagrams are left or right Kan extensions.
Recall that a functor $u\colon A\to B$ is called a \textbf{sieve} if it is fully faithful, and for any morphism $b\to u(a)$ in $B$, there exists an $a'\in A$ with $u(a')=b$.  There is a dual 
notion of a \textbf{cosieve}.
As observed above, left or right Kan extension along a sieve or cosieve is fully faithful.

\begin{lem}[{\cite[Prop.~1.23]{groth:ptstab}}]\label{lem:extbyzero}
  If $u\colon A\to B$ is a sieve and \D is a derivator,  a diagram $X\in\D(B)$ is in the essential image of $u_*$ if and only if $X_b \in\D(\bbone)$ is a terminal object for all $b\notin u(A)$.
  Dually, if $u$ is a cosieve, $X\in\D(B)$ is in the essential image of $u_!$ if and only if $X_b$ is an initial object for all $b\notin u(A)$.
\end{lem}

\begin{rmk}\label{rmk:ebziso}
  In particular, if $u\colon A\to B$ is
  a sieve and we have $X,Y\in\D(B)$ such that $X_b$ and $Y_b$ are
  terminal for $b\notin u(A)$, and moreover $u^*X \cong u^*Y$, then
  $X \cong u_!u^*X \cong u_!u^*Y \cong Y$.  This fact and its
  dual are very convenient, because one of the trickiest parts of
  working with derivators is that coherent diagrams which ``look the
  same'' (have the same underlying diagram) may not be
  isomorphic.  In the context of the inclusion of a (co)sieve, \autoref{lem:extbyzero} says that if
  the ``nontrivial parts'' of two coherent diagrams are isomorphic,
  then the entire diagrams are isomorphic.
\end{rmk}

Our second detection lemma is a version of the familiar theorem from category theory that limits and colimits in functor categories may be computed pointwise.

\begin{lem}[{\cite[Corollary~2.6]{groth:ptstab}}]\label{lem:shiftedkan}
  If $u\colon A\to B$ is fully faithful, then $X\in\shift\D C(B)$ lies in the essential image of $u_!$ (with respect to $\shift\D C$) if and only if for each $c\in C$, the diagram $X_c \in \D(B)$ lies in the essential image of $u_!$ (with respect to \D).
\end{lem}

We now give a criterion to detect when sub-diagrams of a Kan extension are ``colimiting cocones'', generalizing a theorem of~\cite{franke:triang}.
For any category $A$, let $A^\rhd$ be the result of freely adjoining a new terminal object to $A$.
Call the new object $\infty$ and the inclusion $i\colon  A\into A^\rhd$.
Then the square 
\begin{equation}
  \vcenter{\xymatrix{
      A\ar@{=}[r]\ar[d] \drtwocell\omit &
      A\ar[d]^i\\
      \bbone\ar[r]_-\infty &
      A^\rhd.
    }}\label{eq:colimcocone}
\end{equation}is homotopy exact as a special case of a comma square.
Thus, left Kan extensions from $A$ into $A^\rhd$ are an alternative way to compute and characterize colimits over $A$.
We may refer to a coherent diagram in the image of $i_!$ as a \textbf{colimiting cocone}.  

The proof of the following lemma is an immediate generalization of~\cite[Prop.~3.10]{groth:ptstab}.
Its hypotheses may seem technical, but in practice, this is the lemma we reach for most often when it seems ``obvious'' that a certain cocone is colimiting.

\begin{lem}\label{lem:detectionplus}
  Let $A\in\cCat$, and let $u\colon C\to B$ and $v\colon A^\rhd\to B$ be functors.
  Suppose that there is a full subcategory $B'\subseteq B$ such that
  \begin{itemize}
  \item $u(C) \subseteq B'$ and $v(\infty) \notin B'$;
  \item $vi(A) \subseteq B'$; and 
  \item the functor $A \to B' / v(\infty)$ induced by $v$ has a left adjoint.
  \end{itemize}
  Then for any derivator \D and any $X\in\D(C)$, the diagram $v^* u_! X$ is in the essential image of $i_!$.
  In particular, $(v^* u_! X)_\infty$ is the colimit of $i^* v^* u_! X$.
\end{lem}
\begin{proof}
  We want to show that the mate-transformation associated to the square
  \[\vcenter{\xymatrix@-.5pc{
      A\ar[r]^{v i}\ar[d]_i &
      B\ar@{=}[d]\\
      A^\rhd\ar[r]_v &
      B
    }}\]
  is an isomorphism when evaluated at $u_! X$.
  By~\cite[Lemma~1.21]{groth:ptstab}, it suffices to show this for the pasted square
  \[\vcenter{\xymatrix@-.5pc{
      A\ar[r]\ar[d] \drtwocell\omit &
      A\ar[r]^{v i}\ar[d]_i &
      B\ar@{=}[d]\\
      \bbone \ar[r]_-\infty &
      A^\rhd\ar[r]_v &
      B.
    }}\]
  But this square is also equal to the pasting composite
  \[\vcenter{\xymatrix@-.5pc{
      A\ar[r]\ar[d]  &
      B'/v(\infty)\ar[r]\ar[d] \drtwocell\omit &
      B' \ar[r]\ar[d] &
      B\ar@{=}[d]\\
      \bbone\ar@{=}[r] &
      \bbone \ar[r]_{v(\infty)} &
      B\ar@{=}[r] &
      B.
    }}
  \]
  Now the mates associated to each of these three squares are isomorphisms: the left-hand square by the fact that right adjoints are homotopy final, the middle one by (Der4), and the right-hand one because $B'\into B$ is fully faithful. 
\end{proof}

Franke's version of this was the special case for cocartesian squares.
Let $\Box$ denote the category $\bbtwo\times\bbtwo$
\[\vcenter{\xymatrix@-.5pc{
    (0,0)\ar[r]\ar[d] &
    (0,1)\ar[d]\\
    (1,0)\ar[r] &
    (1,1).
  }}\]
Let $\ulcorner$ and $\lrcorner$ denote the full subcategories $\Box \setminus \{(1,1)\}$ and  $\Box \setminus \{(0,0)\}$, respectively, with inclusions $i_\ulcorner\colon \mathord{\ulcorner} \into\Box$ and $i_\lrcorner\colon \mathord{\lrcorner} \into\Box$.
Since $i_\ulcorner$ and $i_\lrcorner$ are fully faithful, so are~$(i_\ulcorner)_!$ and $(i_\lrcorner)_*$.

\begin{defn}
  A coherent diagram $X\in\D(\Box)$ is \textbf{cartesian} if it is in the essential image of $(i_\lrcorner)_*$, and \textbf{cocartesian} if it is in the essential image of $(i_\ulcorner)_!$.
\end{defn}

Taking $A=\ulcorner$ in~\eqref{eq:colimcocone} implies that if $X\in \D(\ulcorner)$ looks like $(y \ot x \to z)$ and $w = (\pi_\ulcorner)_!(X)$ is its pushout, then there is a cocartesian square
\begin{equation}\label{eq:cocartpo}
  \vcenter{\xymatrix@-.5pc{
      x\ar[r]\ar[d] &
      z\ar[d]\\
      y\ar[r] &
      w
      }}
\end{equation}
and conversely, if there is a cocartesian square~\eqref{eq:cocartpo}, then $w\cong (\pi_\ulcorner)_!(X)$.

\begin{rmk}\label{rmk:cocart-cornerdet}
  If we have two cocartesian squares $X,Y\in \D(\Box)$ such that
  $(i_\ulcorner)^*X \cong (i_\ulcorner)^*Y$, then in fact $X\cong Y$,
  and in particular $X_{1,1} \cong Y_{1,1}$.
\end{rmk}

Since cartesian and cocartesian squares play an essential role in the theory of pointed and stable derivators, it is useful to identify them in larger diagrams.
This is the purpose of Franke's lemma, which we can now derive.

\begin{lem}\label{lem:detection}
  Suppose $u\colon C\to B$ and $v\colon \Box \to B$ are functors, with $v$ injective on objects, and let $b = v(1,1)\in B$.
  Suppose furthermore that $b\notin u(C)$, and that the functor $\mathord\ulcorner \to (B\setminus b)/b$ induced by $v$ has a left adjoint.
  Then for any derivator \D and any $X\in \D(C)$, the square $v^* u_! X$ is cocartesian.
\end{lem}
\begin{proof}
  Since $\Box \cong (\mathord{\ulcorner})^\rhd$, we can apply \autoref{lem:detectionplus} with $A=\ulcorner$ and $B' = B\setminus b$.
  (Injectivity of $v$ is needed to ensure that $v(A)\subseteq B'$.)
\end{proof}

Franke's lemma immediately implies the usual ``pasting lemma'' for cocartesian squares.
Let $\boxbar$ denote the category $\bbtwo\times\bbthree$
\[\vcenter{\xymatrix@-.5pc{
    (0,0)\ar[r]\ar[d] &
    (0,1)\ar[r]\ar[d] &
    (0,2)\ar[d]\\
    (1,0)\ar[r] &
    (1,1)\ar[r] &
    (1,2).
  }}
\]
Let $\iota_{j k}$ denote the functor $\Box\to\boxbar$ induced by the identity of $\bbtwo$ on the first factor and the functor $\bbtwo\to\bbthree$ on the second factor which sends $0$ to $j$ and $1$ to $k$.

\begin{cor}[{\cite[Prop.~3.13]{groth:ptstab}}]\label{thm:pasting}
  If $X\in\D(\boxbar)$ is such that $\iota_{01}^*X$ is cocartesian, then $\iota_{02}^*X$ is cocartesian if and only if $\iota_{12}^*X$ is cocartesian.
\end{cor}
\begin{proof}
  Let $A$ be the full subcategory 00-01-02-10 of $\boxbar$, with $j\colon A\into\boxbar$ the inclusion.
  Then \autoref{lem:detection} implies that $\iota_{01}^*j_!Y$ and $\iota_{02}^*j_!Y$ and $\iota_{12}^*j_!Y$ are cocartesian for any $Y\in\D(A)$.
  Thus, for $X\in\D(\boxbar)$ with $\iota_{01}^*X$ cocartesian, it will suffice to show that cocartesianness of $\iota_{02}^*X$ and of $\iota_{12}^*X$ each imply that $\epsilon\colon j_!j^*X\to X$ is an isomorphism.
  Since $j$ is fully faithful, by (Der2) it suffices to check this at (1,1) and (1,2).
  However, cocartesianness of $\iota_{01}^*X$ implies that $\epsilon_{11}$ is an isomorphism, while cocartesianness of $\iota_{02}^*X$ and of $\iota_{12}^*X$ each imply that $\epsilon_{12}$ is an isomorphism.
\end{proof}

Here is another useful consequence of the general form of \autoref{lem:detectionplus}.

\begin{cor}\label{thm:coprod-pushout}
  Coproducts in a derivator are the same as pushouts over the initial object.
  More precisely, for any objects $x$ and $y$ there is a cocartesian square
  \[\vcenter{\xymatrix@-.5pc{
      \emptyset \ar[r]\ar[d] &
      x\ar[d]\\
      y\ar[r] &
      x\sqcup y.
    }}
  \]
\end{cor}
\begin{proof}
  Taking $B=\Box$ and $C = B' = A = \{(1,0),(0,1)\}$ and $A^\rhd = \mathord\lrcorner$ in \autoref{lem:detectionplus}, with $X = (x,y) \in \D(C) \cong \D(\bbone) \times \D(\bbone)$, yields the desired square.
  Its lower-right corner is $x\sqcup y$ by \autoref{lem:detectionplus}, and its upper-left corner is initial by (Der4),
  and it is cocartesian since the left Kan extension from $C$ to $B$ factors through $\ulcorner$.
\end{proof}

Finally, the following lemma says that squares which are constant in one direction are (co)cartesian.
From now on we will use this observation without comment.

\begin{lem}[{\cite[Prop.~3.12(2)]{groth:ptstab}}]\label{thm:trivial-pushout}
  Let $\pi_\bbtwo\colon \Box \to \bbtwo$ denote a projection (either one).
  Then any square in the image of $(\pi_\bbtwo)^*$ is cartesian and cocartesian.
\end{lem}

\section{Pointed derivators and stable derivators}
\label{sec:pointedstable}

In this section we discuss pointed and stable derivators.
Parts of this section are from~\cite{groth:ptstab}, but we also introduce some convenient new results.

\begin{defn}
  A derivator \D is \textbf{pointed} if the category $\D(\bbone)$ has a zero object (an object which is both initial and terminal).
\end{defn}

Since $\pi_A^*\colon \D(\bbone) \to\D(A)$ is both a left and a right adjoint, it preserves zero objects.
Hence, in a pointed derivator each category $\D(A)$ also has a zero object.

\begin{egs}
  A complete and cocomplete category \bC is pointed if and only if $y(\bC)$ is so.
  If a model category or $(\infty,1)$-category is pointed, then so is its homotopy derivator.
  Finally, if \D is pointed, so are $\shift{\D}{B}$ and $\D\op$.
\end{egs}

\autoref{lem:extbyzero} is especially important for pointed derivators, in which case its two characterizations become identical since initial and terminal objects are the same.
Thus, in this case, when $u$ is a sieve we refer to $u_*$ as an \textbf{extension by zero functor}, and similarly for $u_!$ when $u$ is a cosieve.

In a pointed derivator \D,  the \textbf{suspension} functor $\Sigma\colon \D(\bbone) \to\D(\bbone)$ is the composite
\[ \D(\bbone) \xto{(0,0)_*} \D(\mathord{\ulcorner}) \xto{(i_\ulcorner)_!} \D(\Box) \xto{(1,1)^*} \D(\bbone). \]
Since $(0,0)$ is a sieve in $\ulcorner$, the functor ${(0,0)_*}$ is an extension by zero; thus for any $x\in\D(\bbone)$ we have a cocartesian square of the form
\begin{equation}
  \vcenter{\xymatrix@-.5pc{
      x\ar[r]\ar[d] &
      0\ar[d]\\
      0\ar[r] &
      \Sigma x.
    }}\label{eq:susp}
\end{equation}
More generally, any cocartesian square of the form
\begin{equation}
\vcenter{\xymatrix@-.5pc{
    x\ar[r]\ar[d] &
    0\ar[d]\\
    0\ar[r] &
    w
  }}\label{eq:cocart}
\end{equation}
induces a canonical isomorphism $w\cong \Sigma x$.
Note that by \autoref{rmk:ebziso} and \autoref{rmk:cocart-cornerdet}, any two such cocartesian squares containing the same object $x$ are isomorphic.

It is very important to note that  if we restrict a cocartesian square~\eqref{eq:cocart} along the automorphism $\sigma\colon\Box\to\Box$ which swaps $(0,1)$ and $(1,0)$, we obtain a \emph{different} cocartesian square (with the same underlying diagram), and hence a \emph{different} isomorphism $w\cong \Sigma x$.
The relationship between the two is the following.

\begin{lem}
  In any pointed derivator, $\Sigma x$ is a cogroup object, and the composite $\Sigma x \toiso w \toiso \Sigma x$ of the two isomorphisms arising from a cocartesian square~\eqref{eq:cocart} and its $\sigma$-transpose gives the ``inversion'' morphism of $\Sigma x$.
\end{lem}
\begin{proof}
  In~\cite[Prop.~4.12]{groth:ptstab} this is proven under the additional assumption that the derivator is additive, which will always be the case in this paper (see \autoref{thm:additive}).
  A different proof which works more generally can be found in~\cite[\S VI.3]{heller:htpythys}.
\end{proof}

We generally write this cogroup structure additively, and thus denote this morphism by ``$-1$''.

\begin{rmk}
  This may seem strange, but it is not really a new sort of phenomenon.
  Already in ordinary category theory, a universal property is not merely a property of an object, but of that object equipped with extra data, and changing the data can give the same object the same universal property in more than one way.
  For instance, a cartesian product $A\times A$ comes with two projections $\pi_1,\pi_2\colon A\times A\toto A$ exhibiting it as a product of $A$ and $A$, whereas switching these two projections exhibits the same object as a product of $A$ and $A$ in a different way.
  In that case, the induced automorphism of $A\times A$ is the symmetry, $(a,b) \mapsto (b,a)$.
  In the case of suspensions, the ``universal property data'' consists of a cocartesian square~\eqref{eq:susp}, and transposing the square is analogous to switching the projections.
\end{rmk}

The suspension functor of $\D\op$ is called the \textbf{loop space} functor of \D and denoted~$\Omega$.
By definition, $\Omega x$ comes with a coherent diagram of shape $\Box$ in $\D\op$.
In~\D, this is a diagram of shape $\Box\op$, hence looks like
\begin{equation}
\vcenter{\xymatrix@-.5pc{
    x\ar@{<-}[r]\ar@{<-}[d] &
    0\ar@{<-}[d]\\
    0\ar@{<-}[r] &
    \Omega x.
  }}\label{eq:cartop}
\end{equation}
Restricting this along the isomorphism $\tau\colon\Box \toiso \Box\op$ which fixes $(0,1)$ and $(1,0)$ and exchanges $(0,0)$ with $(1,1)$, we obtain a cartesian square in \D of the form
\begin{equation}
\vcenter{\xymatrix@-.5pc{
    \Omega x\ar[r]\ar[d] &
    0\ar[d]\\
    0\ar[r] &
    x.
  }}\label{eq:cart}
\end{equation}

It may seem terribly pedantic to distinguish between diagrams of shape $\Box$ and $\Box\op$, but we find that it helps avoid confusion with minus signs.
In particular, if instead of the isomorphism $\tau$ we used the isomorphism $\tau\sigma$ (which ``rotates''~\eqref{eq:cartop} to make it look like~\eqref{eq:cart}), we would obtain a \emph{different} cartesian square of shape~\eqref{eq:cart} in \D.
The difference would, again, be the inversion map $-1\colon\Omega x \to \Omega x$.

\begin{lem}[{\cite[Prop.~3.17]{groth:ptstab}}]\label{lem:susploopadj}
  There is an adjunction $\Sigma\dashv\Omega$.
\end{lem}

The \textbf{cofiber functor} $\cof\colon \D(\bbtwo)\to\D(\bbtwo)$ in a pointed derivator is the composite
\[ \D(\bbtwo) \xto{(0,-)_*} \D(\ulcorner) \xto{(i_{\ulcorner})_!} \D(\Box) \xto{(-,1)^*} \D(\bbtwo). \]
Here $(0,-)\colon \bbtwo\to\ulcorner$ indicates the inclusion as the objects with first coordinate $0$, and similarly for $(-,1)\colon \bbtwo\to\Box$.
Since $(0,-)$ is a sieve, $(0,-)_*$ is an extension by zero; thus by stopping after the first two functors we have a cocartesian square
\begin{equation}\label{eq:cofiber}
  \vcenter{\xymatrix@-.5pc{
      x \ar[r]^f \ar[d] & y \ar[d]^{\cof(f)}\\
      0 \ar[r] & z.
    }}
\end{equation}
By Remarks~\ref{rmk:ebziso} and~\ref{rmk:cocart-cornerdet}, any two cocartesian squares~\eqref{eq:cofiber} with the same underlying object $(x\xto{f} y)$ of $\D(\bbtwo)$ are isomorphic.

\begin{rmk}
  In a \emph{strong} pointed derivator, every morphism in $\D(\bbone)$ underlies some object of $\D(\bbtwo)$.
  Thus, we can construct ``the'' cofiber of any morphism in $\D(\bbone)$ by first lifting it to an object of $\D(\bbtwo)$.
  By \autoref{rmk:der5}, the result is independent of the chosen lift, up to \emph{non-unique isomorphism}.
\end{rmk}

Dually, the \textbf{fiber functor} $\fib\colon \D(\bbtwo)\to\D(\bbtwo)$ is the cofiber functor of $\D\op$, which can be identified with the composite
\[ \D(\bbtwo) \xto{(-,1)_!} \D(\lrcorner) \xto{(i_{\lrcorner})_*} \D(\Box) \xto{(0,-)^*} \D(\bbtwo) \]
so that we have a cartesian square
\begin{equation}
  \xymatrix@R=1.5pc{w \ar[r]^{\fib(f)} \ar[d] & x \ar[d]^{f}\\
    0 \ar[r] & y.
  }
\end{equation}

\begin{lem}[{\cite[Prop.~3.20]{groth:ptstab}}]\label{lem:coffibadj}
  There is an adjunction $\cof\dashv\fib$.
\end{lem}

In a pointed derivator \D, we define a \textbf{cofiber sequence} to be a coherent diagram of shape $\boxbar=\bbtwo\times\bbthree$ in which both squares are cocartesian and whose $(0,2)$- and $(1,0)$-entries are zero objects:
\begin{equation}
  \vcenter{\xymatrix@-.5pc{
      x\ar[r]^f\ar[d] &
      y\ar[r]\ar[d]^g &
      0\ar[d]\\
      0\ar[r] &
      z\ar[r]_h &
      w
    }}\label{eq:cofiberseq}
\end{equation}
Suitable combinations of Kan extensions give a functorial construction of cofiber sequences $\D(\bbtwo)\to\D(\boxbar)$. 
This functor induces an equivalence onto the full subsategory of $\D(\boxbar)$ spanned by the cofiber sequence. Thus, for derivators  a morphism is equivalent to 
 its cofiber sequence, and there are variants of this for iterated cofiber sequences, fiber sequences, and similar such constructions.

Recall that $\iota_{j k}$ denotes the functor $\Box\to\boxbar$ induced by the identity of $\bbtwo$  on the first factor and the functor $\bbtwo\to\bbthree$ on the second factor which sends $0$ to $j$ and $1$ to $k$.
Then a cofiber sequence is an $X\in\D(\boxbar)$ such that $X_{(0,2)}$ and $X_{(1,0)}$ are zero objects and $\iota_{01}^*X$ and $\iota_{12}^*X$ are cocartesian.
By \autoref{thm:pasting}, $\iota_{02}^*X$ is also cocartesian, and therefore induces an isomorphism $w\cong \Sigma x$.
(As always, $\sigma^*\iota_{02}^*X$ is also cocartesian, but would induce the opposite isomorphism $w\cong \Sigma x$.)
Of course, by restricting to the two cocartesian squares, we also obtain canonical isomorphisms $g\cong \cof(f)$ and $h\cong \cof(g)$.

As suggested in \cite[Remarque 2.1.62]{ayoub:six}, the identification of $w$ with $\Sigma x$ can also be made functorial.

\begin{lem}\label{thm:cof-cubed}
  The functor $\cof^3\colon \D(\bbtwo)\to\D(\bbtwo)$ is naturally isomorphic to the suspension functor $\Sigma$ of $\D^\bbtwo$.
\end{lem}

\begin{proof}
  Let $A$ be the full subcategory of $\bbthree\times\bbthree$ which omits $(2,0)$.
  Using a combination of extension by zero functors and left Kan extensions, we have a functor $\D(\bbtwo)\to\D(A)$ which sends $x\xto{f}y$ to a diagram of the following form
  \begin{equation}\label{eq:cof3}
    \vcenter{\xymatrix@-.7pc{
      x \ar[r]^f \ar[d] &
      y \ar[d]_g \ar[r] &
      0_2 \ar[d]\\
      0_1 \ar[r] &
      z \ar[r]^h \ar[d] &
      w \ar[d]^k\\
      & 0_3 \ar[r] &
      v.}}
  \end{equation}
  (Ignore the subscripts for now; all objects denoted $0_k$ are zero objects.)
  \autoref{lem:detection} implies that all squares and rectangles in this diagram are cocartesian.
  Thus we have a canonical identification of $g\in\D(\bbtwo)$ with $\cof(f)$, and similarly of $h$ and $k$ with $\cof^2(f)$ and $\cof^3(f)$.

  Now let $C = \bbtwo^3$ be the shape of a cube, and let $q\colon C\to A$ be the functor such that $q^*$ of~\eqref{eq:cof3} has the following form
  \[\xymatrix@-1pc{
    x \ar[rr]^f \ar[dr] \ar[dd] &&
    y \ar[dr] \ar'[d][dd] \\
    & 0_2 \ar[rr] \ar[dd] &&
    0_2 \ar[dd]\\
    0_1 \ar[dr] \ar'[r][rr] &&
    0_3 \ar[dr]\\
    & w \ar[rr]_k && v.
  }\]
  Here the subscripts match those in~\eqref{eq:cof3} to indicate the definition of $q$ precisely.
  This cube may be regarded as a coherent square in $\shift\D\bbtwo$ (with the $\bbtwo$-direction going left-to-right).
  Moreover, since its left and right faces are cocartesian in \D, applying \autoref{lem:shiftedkan} with $A= \ulcorner$, $B= \Box$ and $C=\bbtwo$ we see the cube is cocartesian in $\shift\D\bbtwo$.
  Thus, it naturally identifies $k\cong\cof^3(f)$ with $\Sigma(f)$.
\end{proof}

\begin{rmk}\label{rmk:rotating-signs}
  The identification of $k$ with $\Sigma(f)$ in the proof of \autoref{thm:cof-cubed} mandates that we identify $w$ and $v$ with $\Sigma x$ and $\Sigma y$ using the cocartesian squares
  \begin{equation}
    \vcenter{\xymatrix{
        x\ar[r]\ar[d] &
        0_2\ar[d]\\
        0_1\ar[r] &
        w
      }} \qquad\text{and}\qquad
    \vcenter{\xymatrix{
        y\ar[r]\ar[d] &
        0_2\ar[d]\\
        0_3\ar[r] &
        v
      }}
  \end{equation}
  respectively.
  Of course, if we were to instead use the transpose of \emph{one} of these squares, then $k$ would instead be identified with $-\Sigma f$.
  This is exactly what happens in the proof in~\cite[Theorem~4.16]{groth:ptstab} that distinguished triangles can be ``rotated'' (axiom (T2) of a triangulated category).
\end{rmk}

We define a \textbf{fiber sequence} in $\D$ to be a cofiber sequence in $\D\op$.
Thus, it is a diagram of shape $\boxbar\op$ in \D, which looks like
\[\vcenter{\xymatrix@-.5pc{
    z\ar@{<-}[r]\ar@{<-}[d] &
    y\ar@{<-}[r]\ar@{<-}[d] &
    0\ar@{<-}[d]\\
    0\ar@{<-}[r] &
    x\ar@{<-}[r] &
    w.
  }}
\]
By restricting along the ``rotation'' isomorphism $\rho\colon\boxbar \toiso \boxbar\op$, we can draw this as a diagram of shape $\boxbar$ in \D
\begin{equation}
\vcenter{\xymatrix@-.5pc{
    w\ar[r]\ar[d] &
    x\ar[r]\ar[d] &
    0\ar[d]\\
    0\ar[r] &
    y\ar[r] &
    z
  }}\label{eq:fiberseq}
\end{equation}
in which both squares are cartesian.
As before, it follows that the outer rectangle is also cartesian, and hence we can identify $w$ with the loop space object $\Omega z$.
Note, though, that the isomorphism $\Box\toiso \Box\op$ induced on the outer rectangle by $\rho$ is $\tau\sigma$, not $\tau$.
In fact, there is no isomorphism $\boxbar \toiso \boxbar\op$ which induces $\tau$ on the outer rectangles.

We now turn to stable derivators.
In contrast to pointedness, which at least has nontrivial examples in the representable case (even if suspensions and loops are generally not very interesting there), stability is entirely a homotopical notion: the only stable represented derivator is $y(\bbone)$.
For now, we give three versions of the definition; we will see in \S\ref{sec:characterization} that they are actually equivalent.

\begin{defn}
  Let \D be a pointed derivator.
  \begin{enumerate}
  \item \D is \textbf{stable} if a coherent square
    \begin{equation}\label{eq:stabsq}
      \xymatrix@-.5pc{x \ar[r] \ar[d] & y \ar[d] \\ z \ar[r] & w }
    \end{equation}
    is cartesian if and only if it is cocartesian.
  \item \D is \textbf{cofiber-stable} if this is true under the additional assumption that $z$ is a zero object.
  \item \D is \textbf{$\Sigma$-stable} if this is true under the additional assumption that $y$ and $z$ are both zero objects.
  \end{enumerate}
\end{defn}

A square which is both cartesian and cocartesian is called~\textbf{bicartesian}.
In particular, in a $\Sigma$-stable derivator, the bicartesian square~\eqref{eq:susp} induces an isomorphism $x \cong \Omega\Sigma x$, and similarly~\eqref{eq:cart} induces an isomorphism $\Sigma \Omega x \cong x$.
Thus, the adjunction $\Sigma\dashv \Omega$ is an equivalence.
Similarly, $\cof\dashv\fib$ is an equivalence in a cofiber-stable derivator.

\begin{eg}
  A \emph{stable model category} is, by definition, a pointed model category whose homotopy derivator is $\Sigma$-stable.
  However, by~\cite[Remark~7.1.12]{hovey:modelcats} such a derivator is in fact stable.
  (In \autoref{thm:stable} we will generalize this fact to all derivators.)
  Thus, any stable model category gives rise to a stable derivator.
\end{eg}

\begin{eg}
  As defined in~\cite[1.1.1.9]{lurie:ha}, a \emph{stable $(\infty,1)$-category} is a pointed $(\infty,1)$-category with fibers and cofibers and in which fiber sequences agree with cofiber sequences.
  Thus, a complete and cocomplete $(\infty,1)$-category is stable if and only if its homotopy derivator is cofiber-stable.
\end{eg}

One of the basic facts about stable derivators is the following.

\begin{lem}\label{thm:additive}
  Any $\Sigma$-stable derivator is \emph{additive} in the sense that finite products and coproducts coincide naturally.
\end{lem}

  The proof of this fact in~\cite[Prop.~4.7]{groth:ptstab} uses ``full'' stability.  In the next section we will need the fact that additivity requires only $\Sigma$-stability  
  so  we sketch an alternative proof.
\begin{proof}[Sketch of proof]
  It is well-known that a category with products and coproducts is additive when every object has a commutative monoid structure and every morphism is a monoid map.
  Now by~\cite[Lemma~4.11]{groth:ptstab}, any object of the form $\Omega X$ is a monoid; the statement in~\cite{groth:ptstab} assumes stability, but the proof uses only pointedness.
  The construction is natural, so every morphism of the form $\Omega f$ is a monoid map.
  The usual Eckmann--Hilton argument implies that any object of the form $\Omega^2 X$ is a \emph{commutative} monoid.
  However, $\Sigma$-stability implies that every object is of the form $\Omega^2 X$ and every morphism is of the form $\Omega^2 f$.
\end{proof}

In a cofiber-stable derivator, if we have a cofiber sequence
\begin{equation}
  \vcenter{\xymatrix@-.5pc{
      x\ar[r]^f\ar[d] &
      y\ar[r]\ar[d]_g &
      0\ar[d]\\
      0\ar[r] &
      z\ar[r]^h &
      w,
    }}\label{eq:cofiberseq2}
\end{equation}
then we say that the induced string of composable arrows in $\D(\bbone)$
\begin{equation}
  \DT x f y g z {h'} {\Sigma x}\label{eq:dt}
\end{equation}
is a \textbf{distinguished triangle}.
Here $h'$ is the composite $z\xto{h} w \toiso \Sigma x$, the isomorphism being induced by the outer rectangle of~\eqref{eq:cofiberseq2}.
Note that a distinguished triangle is an \emph{incoherent} diagram, i.e.\ an object of $\D(\bbone)^\bbthree$ rather than $\D(\bbthree)$.
As usual, we also extend the term \emph{distinguished triangle} to any such incoherent diagram which is \emph{isomorphic} to one obtained in this way.

\begin{thm}[{\cite[Theorem~4.16]{groth:ptstab}}]\label{thm:triang}
  If \D is a strong, stable derivator, then the suspension functor and distinguished triangles defined above make $\D(\bbone)$ into a \emph{triangulated category} in the sense of Verdier.
\end{thm}

\begin{rmk}
  We need the assumption that \D is strong because the triangulation axioms for $\D(\bbone)$ refer only to morphisms of $\D(\bbone)$ (having no other option), whereas to prove the axioms we need to lift such morphisms to objects of $\D(\bbtwo)$.
  For instance, to extend a morphism $f\colon x\to y$ to a distinguished triangle, we need $f$ to be an object of $\D(\bbtwo)$ so as to be able to extend it to a cofiber sequence. Thus, the strongness is only needed to relate properties of derivators to structure on its values but not for the theory of derivators itself.
\end{rmk}

It is crucial that in passing from cofiber sequences to distinguished triangles, we use the isomorphism $w\cong \Sigma x$ obtained from the outer rectangle of~\eqref{eq:cofiberseq2} and not its $\sigma$-transpose.
In particular, this implies that although \emph{fiber sequences} and \emph{cofiber sequences} essentially coincide in a stable derivator (modulo $\rho^*$), they do not induce the same notion of ``distinguished triangle''.
This is expressed by the following lemma, whose analogue for homotopy categories of stable model categories is well-known (see e.g.~\cite[Theorem~7.1.11]{hovey:modelcats}).

\begin{lem}\label{lem:fibercofiber}
  If \D is a stable derivator, then the distinguished triangles in $\D\op(\bbone)$ are the negatives of those in $\D(\bbone)$.
\end{lem}
Recall that the negative of a triangulation is obtained by negating an odd number of the morphisms in each given distinguished triangle.
\begin{proof}
  Suppose $X\in\D(\boxbar)$ is a cofiber sequence in \D that looks like~\eqref{eq:cofiberseq2}.
  Then since cocartesian squares in \D are also cartesian, $\rho^* X\in\D(\boxbar\op) = \D\op(\boxbar)\op$ is a fiber sequence in \D, i.e.\ a cofiber sequence in $\D\op$.
  It therefore induces a distinguished triangle in $\D\op$, which interpreted in $\D$ looks like
  \begin{equation}
    \xymatrix{w \ar@{<-}[r]^{h} & z \ar@{<-}[r]^{g} & y \ar@{<-}[r]^{f'} & \Omega w.}\label{eq:codistinguished}
  \end{equation}
  Since $\Sigma$ is inverse to $\Omega$, we can turn this around and write it as
  \[ \DT{\Omega w}{f'}{y}{g}{z}{h}{\Sigma\Omega w} \]
  and then use (either) isomorphism $x\cong \Omega w$ to write it as
  \[ \DT{x}{f'}{y}{g}{z}{h}{\Sigma x.} \]
  However, because in~\eqref{eq:codistinguished} $x$ is identified with $\Omega w$ (the suspension of $w$ in $\D\op$) \emph{not} via the outer cartesian rectangle of~\eqref{eq:cofiberseq2}, but its $\sigma$-transpose (since $\rho$ restricts to $\tau\sigma$ on the outer rectangle of $\boxbar$), we have $f'=-f$ and not $f$.
\end{proof}

\section{Mayer-Vietoris sequences}
\label{sec:mv}

The following is our main result.
It was previously known to be true for homotopy categories of stable model categories (see~\cite[Lemma~5.7]{may:traces}). 

\begin{thm}\label{prop:pushout-dt}
  In a cofiber-stable derivator \D, if we have a cocartesian square
  \begin{equation}
    \vcenter{\xymatrix{
        x\ar[r]^f\ar[d]_g &
        y\ar[d]^j\\
        z\ar[r]_k &
        w
      }}
  \end{equation}
  then there is a cocartesian square
  \begin{equation}
    \vcenter{\xymatrix@C=3pc{
        x\ar[r]^-{(f,-g)}\ar[d] &
        y\oplus z\ar[d]^{[j,k]}\\
        0\ar[r] &
        w
      }}
  \end{equation}
  and hence a distinguished triangle
  \[ \DTl{x}{(f,-g)}{y \oplus z}{[j,k]}{w}{}{\Sigma x.}\]
\end{thm}

\begin{proof}
Suppose given a cocartesian square
\begin{equation}
  \vcenter{\xymatrix{
      x\ar[r]^f\ar[d]_g &
      y\ar[d]^j\\
      z\ar[r]_k &
      w.
    }}\label{eq:pushoutdt}
\end{equation}
We will construct a cofiber sequence
\begin{equation}
  \vcenter{\xymatrix@C=3pc{
      y\oplus z\ar[r]^-{[j,k]}\ar[d] &
      w \ar[r] \ar[d] &
      0 \ar[d]\\
      0\ar[r] &
      \Sigma x\ar[r]_-{\Sigma (-f,g)} &
      \Sigma y\oplus \Sigma z
    }}
\end{equation}
and hence a distinguished triangle
\begin{equation}
  \DTl{y \oplus z}{[j,k]}{w}{}{\Sigma x}{\Sigma(-f,g)}{\Sigma y \oplus \Sigma z}.\label{eq:pdtsuspd}
\end{equation}
Rotating backwards (which we can do in any cofiber-stable derivator, where fiber sequences coincide with cofiber sequences) will produce the desired result.
(Recall that rotating a triangle negates the suspended morphism, so that $\Sigma(-f,g)$ becomes $(f,-g)$.)

We begin by restricting from~\eqref{eq:pushoutdt} to obtain a diagram as on the left below.
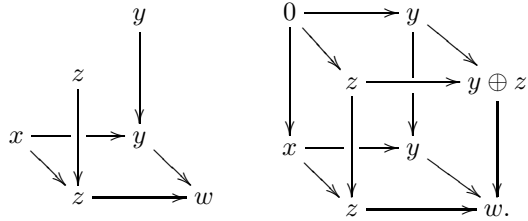
\[\xymatrix@-1pc{
  &&
  y  \ar[dd] \\
  & z  \ar[dd] &&
  \\
  x \ar[dr] \ar'[r][rr] &&
  y \ar[dr]\\
  & z \ar[rr] && w
} \qquad\xymatrix@-1pc{
  0 \ar[rr] \ar[dr] \ar[dd] &&
  y \ar[dr] \ar'[d][dd] \\
  & z \ar[rr] \ar[dd] &&
  y\oplus z \ar[dd]\\
  x \ar[dr] \ar'[r][rr] &&
  y \ar[dr]\\
  & z \ar[rr] && w.
}\]
Then by left Kan extension we obtain a cube as on the right, whose upper and lower squares are cocartesian.
Cocartesianness of the upper square, along with the identification of $0$ and $y\oplus z$, follow as in the proof of \autoref{thm:coprod-pushout}.

Regarding this right-hand cube as an object of $\shift\D\Box(\bbtwo)$, with its domain and codomain being the top square and bottom square respectively, we can construct its cofiber sequence.
The result is a diagram in \D of the ``double hypercube'' shape $\Box\times\boxbar$ shown in \autoref{fig:double-hypercube}.
As before, the subscripts are merely to distinguish different occurrences of the same (or isomorphic) objects.
\begin{figure}
  \centering
  \begin{tikzpicture}[->,xscale=1.3]
    \node (o-a) at (0.6,7.2) {$0$};
    \node (x1) at (3.4,7.2) {$x_1$};
    \node (y1) at (1,5.8) {$y_1$};
    \node (y2) at (3.8,5.8) {$y_2$};
    \node (z1) at (2,5) {$z_1$};
    \node (z2) at (4.8,5) {$z_2$};
    \node (yz) at (2.4,3.8) {$y\oplus z$};
    \node (w) at (5.2,3.8) {$w$};
    \node (o1) at (2,3) {$0_1$};
    \node (o5) at (4.8,3) {$0_5$};
    \node (o2) at (2.4,1.8) {$0_2$};
    \node (sx) at (5.2,1.8) {$\Sigma x$};
    \node (o-b) at (0.6,1.2) {$0$};
    \node (x2) at (3.4,1.2) {$x_2$};
    \node (o3) at (1,0) {$0_3$};
    \node (o4) at (3.8,0) {$0_4$};
    \node (o-c) at (6.2,7.2) {$0$};
    \node (o6) at (6.6,5.8) {$0_6$};
    \node (o7) at (7.6,5) {$0_7$};
    \node (o8) at (8.3,3.8) {$0_8$};
    \node (sy) at (6.6,0) {$\Sigma y$};
    \node (o9) at (6.2,1.2) {$0_9$};
    \node (sz) at (7.6,3) {$\Sigma z$};
    \node (sysz) at (8.3,1.8) {$\Sigma y \oplus \Sigma z$};
    \draw (o-c) -- (o6);
    \draw (x1) -- (o-c);
    \draw (o-c) -- (o7);
    \draw (o6) -- (o8);
    \draw (o7) -- (o8);
    \draw (o9) -- (sy);
    \draw (x2) -- (o9);
    \draw (o4) -- (sy);
    \draw (o-c) -- (o9);
    \draw (z2) -- (o7);
    \draw (o7) -- (sz);
    \draw (sy) -- (sysz);
    \draw (sz) -- (sysz);
    \draw (o8) -- (sysz);
    \draw (o9) -- (sz);
    \draw (o-a) -- (x1); \draw (x1) -- (y2);
    \draw (o-a) -- (y1);
    \draw (z1) -- (yz); \draw (z2) -- (w);
    \draw (o1) -- (o5);
    \draw (o1) -- (o2); \draw (o5) -- (sx);
    \draw (o-b) -- (x2); \draw (o3) -- (o4);
    \draw (o-b) -- (o3); \draw (x2) -- (o4);
    \draw (o-a) -- (z1); \draw (x1) -- (z2); 
    \draw (o-b) -- (o1); \draw (x2) -- (o5); 
    \draw (o-a) -- (o-b); \draw (x1) -- (x2);
    \draw (z1) -- (o1); \draw (z2) -- (o5);
    \draw[white,line width=5pt,-] (sx) -- (sysz);    \draw (sx) -- (sysz);
    \draw[white,line width=5pt,-] (o5) -- (sz);    \draw (o5) -- (sz);
    \draw[white,line width=5pt,-] (w) -- (o8);    \draw (w) -- (o8);
    \draw[white,line width=5pt,-] (z2) -- (o7); \draw (z2) -- (o7);
    \draw[white,line width=5pt,-] (o6) -- (sy);     \draw (o6) -- (sy);
    \draw[white,line width=5pt,-] (o6) -- (o8); \draw (o6) -- (o8);
    \draw[white,line width=5pt,-] (y2) -- (o6); \draw (y2) -- (o6);
    \draw[white,line width=5pt,-] (o2) -- (sx); \draw (o2) -- (sx);
    \draw[white,line width=5pt,-] (o1) -- (o5); \draw (o1) -- (o5);
    \draw[white,line width=5pt,-] (z1) -- (z2); \draw (z1) -- (z2);
    \draw[white,line width=5pt,-] (yz) -- (w); \draw (yz) -- (w);
    \draw[white,line width=5pt,-] (y2) -- (o4); \draw (y2) -- (o4);
    \draw[white,line width=5pt,-] (y1) -- (y2);     \draw (y1) -- (y2);  
    \draw[white,line width=5pt,-] (y1) -- (yz);    \draw (y1) -- (yz); 
    \draw[white,line width=5pt,-] (y2) -- (w);    \draw (y2) -- (w); 
    \draw[white,line width=5pt,-] (o4) -- (sx);   \draw (o4) -- (sx);
    \draw[white,line width=5pt,-] (o3) -- (o2);   \draw (o3) -- (o2);
    \draw[white,line width=5pt,-] (y1) -- (o3);    \draw (y1) -- (o3); 
    \draw[white,line width=5pt,-] (yz) -- (o2);   \draw (yz) -- (o2);
    \draw[white,line width=5pt,-] (w) -- (sx);   \draw (w) -- (sx);
  \end{tikzpicture}
  \caption{The double hypercube}
  \label{fig:double-hypercube}
\end{figure}
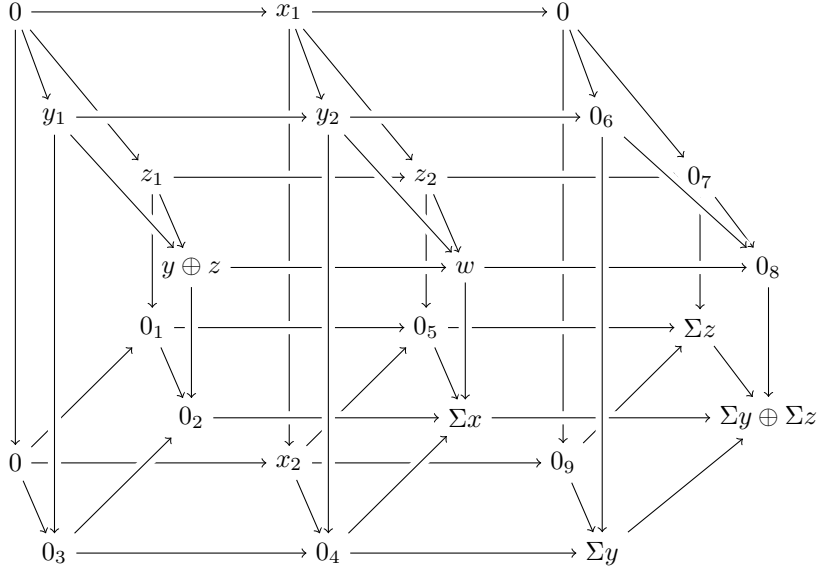
Since left Kan extensions in $\shift\D\Box$ are pointwise by \autoref{lem:shiftedkan}, all the squares in \autoref{fig:double-hypercube} whose sides are parallel to the axes are cocartesian.
This allows us to identify the objects labeled $0_4$, $0_5$, and $0_9$ as zero objects, and the objects labeled $\Sigma y$ and $\Sigma z$ as the suspensions indicated, by using the squares
\begin{equation}
  \vcenter{\xymatrix{
      y_1\ar[r]\ar[d] &
      0_6\ar[d]\\
      0_3\ar[r] &
      \Sigma y
      }}
    \qquad\text{and}\qquad \vcenter{\xymatrix{
      z_1\ar[r]\ar[d] &
      0_7\ar[d]\\
      0_1\ar[r] &
      \Sigma z.
      }}
\end{equation}

Next, since the diagram of categories
\begin{equation}
  \vcenter{\xymatrix{
      \mathord\ulcorner\times\mathord\ulcorner\ar[r]\ar[d] &
      \mathord\ulcorner\times\Box\ar[d]\\
      \Box\times\mathord\ulcorner\ar[r] &
      \Box\times\Box
      }}
\end{equation}
commutes, the pushout in $\shift\D\Box$ of a diagram of cocartesian squares is cocartesian.
Since the constant zero square is also cocartesian, all the squares occurring as objects in the cofiber sequence in $\shift\D\Box$ are cocartesian.
In particular, the squares in \autoref{fig:double-hypercube} labeled as
\begin{equation}\label{eq:sigma-x}
  \vcenter{\xymatrix{
      x_2\ar[r]\ar[d] &
      0_4\ar[d]\\
      0_5\ar[r] &
      \Sigma x
      }}
    \qquad\text{and}\qquad
  \vcenter{\xymatrix{
      0_9\ar[r]\ar[d] &
      \Sigma y\ar[d]\\
      \Sigma z\ar[r] &
      \Sigma y \oplus \Sigma z
      }}
\end{equation}
are cocartesian, allowing us to identify their lower-right corners as indicated.
Note that the latter is consistent with our identifications of $\Sigma y$ and $\Sigma z$ and the isomorphism $\Sigma(y\oplus z) \cong \Sigma y\oplus \Sigma z$.
By contrast, there is no canonical reason to choose the former over its transpose; our choice determines the appearance of $(-f,g)$ in the cofiber sequence rather than $(f,-g)$.

Inside \autoref{fig:double-hypercube} we see our desired cofiber sequence
\begin{equation}
  \vcenter{\xymatrix@C=3pc{
      y\oplus z\ar[r]\ar[d] &
      w \ar[r] \ar[d] &
      0_8 \ar[d]\\
      0_2\ar[r] &
      \Sigma x\ar[r] &
      \Sigma y\oplus \Sigma z.
    }}
\end{equation}
It is obvious that the map $y\oplus z \to w$ in this cofiber sequence is ${[j,k]}$; thus it remains only to identify the map $\Sigma x \to \Sigma y \oplus \Sigma z$ with $\Sigma(-f,g) = (-\Sigma f, \Sigma g)$.
Since \D is additive by \autoref{thm:additive}, it will suffice to show that the composites of this map with the two projections $\Sigma y \oplus \Sigma z \to \Sigma y$ and $\Sigma y \oplus \Sigma z \to \Sigma z$ are $-\Sigma f$ and $\Sigma g$, respectively.

We will show that the composite with the second projection $\Sigma y \oplus \Sigma z \to \Sigma z$ is $\Sigma g$, and then indicate how the other argument differs.
For this, we will extend \autoref{fig:double-hypercube} to a larger coherent diagram which includes the second projection: this diagram is shown in \autoref{fig:extended-hypercube} (we have neglected to draw the left half of the diagram in \autoref{fig:extended-hypercube} for brevity).
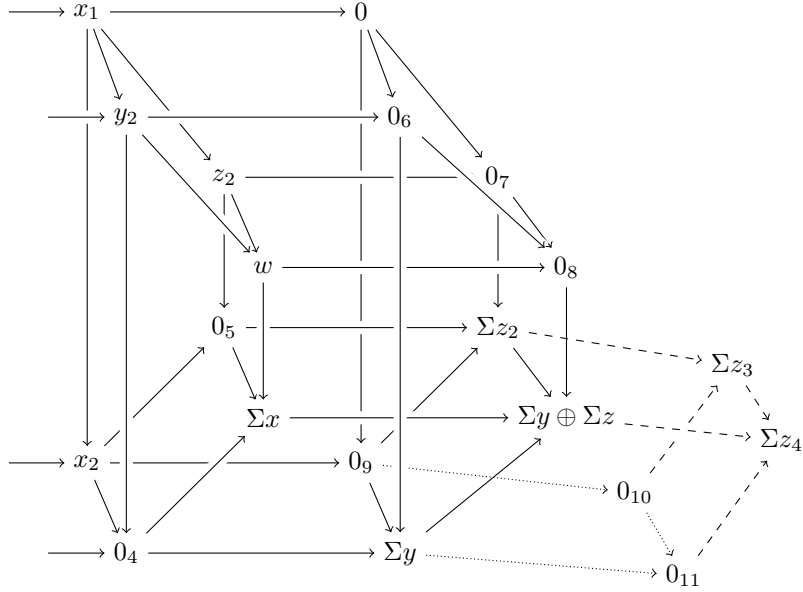
\begin{figure}
  \centering
  \begin{tikzpicture}[->,xscale=1.3]
    \node (x1) at (3.4,7.2) {$x_1$};
    \node (y2) at (3.8,5.8) {$y_2$};
    \node (z2) at (4.8,5) {$z_2$};
    \node (w) at (5.2,3.8) {$w$};
    \node (o5) at (4.8,3) {$0_5$};
    \node (sx) at (5.2,1.8) {$\Sigma x$};
    \node (x2) at (3.4,1.2) {$x_2$};
    \node (o4) at (3.8,0) {$0_4$};
    \node (o-c) at (6.2,7.2) {$0$};
    \node (o6) at (6.6,5.8) {$0_6$};
    \node (o7) at (7.6,5) {$0_7$};
    \node (o8) at (8.3,3.8) {$0_8$};
    \node (sy) at (6.6,0) {$\Sigma y$};
    \node (o9) at (6.2,1.2) {$0_9$};
    \node (sz) at (7.6,3) {$\Sigma z_2$};
    \node (sysz) at (8.3,1.8) {$\Sigma y \oplus \Sigma z$};
    \node (o10) at (9,0.8) {$0_{10}$};
    \node (o11) at (9.5,-0.3) {$0_{11}$};
    \node (sz3) at (10,2.5) {$\Sigma z_3$};
    \node (sz4) at (10.5,1.5) {$\Sigma z_4$};
    \draw (o-c) -- (o6);
    \draw (x1) -- (o-c);
    \draw (o-c) -- (o7);
    \draw (o6) -- (o8);
    \draw (o7) -- (o8);
    \draw (o9) -- (sy);
    \draw (x2) -- (o9);
    \draw (o4) -- (sy);
    \draw (o-c) -- (o9);
    \draw (z2) -- (o7);
    \draw (o7) -- (sz);
    \draw (sy) -- (sysz);
    \draw (sz) -- (sysz);
    \draw (o8) -- (sysz);
    \draw (o9) -- (sz);
    \draw (x1) -- (y2);
    \draw (z2) -- (w);
    \draw (o5) -- (sx);
    \draw (x2) -- (o4);
    \draw (x1) -- (z2); 
    \draw (x2) -- (o5); 
    \draw (x1) -- (x2);
    \draw (z2) -- (o5);
    \begin{scope}[densely dotted]
      \draw (o9) -- (o10);
      \draw (sy) -- (o11);
      \draw (o10) -- (o11);
    \end{scope}
    \begin{scope}[dashed]
      \draw (sz) -- (sz3);
      \draw (sz3) -- (sz4);
      \draw (o10) -- (sz3);
      \draw (o11) -- (sz4);
    \end{scope}
    \draw[white,line width=5pt,-] (sysz) -- (sz4); \draw[dashed] (sysz) -- (sz4);
    \draw[white,line width=5pt,-] (sx) -- (sysz);    \draw (sx) -- (sysz);
    \draw[white,line width=5pt,-] (o5) -- (sz);    \draw (o5) -- (sz);
    \draw[white,line width=5pt,-] (w) -- (o8);    \draw (w) -- (o8);
    \draw[white,line width=5pt,-] (z2) -- (o7); \draw (z2) -- (o7);
    \draw[white,line width=5pt,-] (o6) -- (sy);     \draw (o6) -- (sy);
    \draw[white,line width=5pt,-] (o6) -- (o8); \draw (o6) -- (o8);
    \draw[white,line width=5pt,-] (y2) -- (o6); \draw (y2) -- (o6);
    \draw[white,line width=5pt,-] (y2) -- (o4); \draw (y2) -- (o4);
    \draw[white,line width=5pt,-] (y2) -- (w);    \draw (y2) -- (w); 
    \draw[white,line width=5pt,-] (o4) -- (sx);   \draw (o4) -- (sx);
    \draw[white,line width=5pt,-] (w) -- (sx);   \draw (w) -- (sx);
    \draw[<-] (x1) -- +(-.8,0);
    \draw[white,line width=5pt,-] (y2) -- +(-.8,0);
    \draw[<-] (y2) -- +(-.8,0);
    \draw[<-] (x2) -- +(-.8,0);
    \draw[<-] (o4) -- +(-.8,0);
  \end{tikzpicture}
  \caption{The extended double hypercube}
  \label{fig:extended-hypercube}
\end{figure}
To obtain \autoref{fig:extended-hypercube} from \autoref{fig:double-hypercube}, first we extend by zero to add the zero objects $0_{10}$ and $0_{11}$ and the dotted arrows.
Then we left Kan extend to add the other two objects and the dashed arrows.

Now by \autoref{lem:detection}, the squares
\begin{equation}
  \vcenter{\xymatrix{
      0_9\ar[r]\ar[d] &
      0_{10}\ar[d]\\
      \Sigma z_2\ar[r] &
      \Sigma z_3
      }}
    \qquad\text{and}\qquad
  \vcenter{\xymatrix{
      \Sigma y\ar[r]\ar[d] &
      \Sigma y\oplus \Sigma z\ar[d]\\
      0_{11}\ar[r] &
      \Sigma z_4
      }}
\end{equation}
appearing in \autoref{fig:extended-hypercube} are cocartesian.
The first allows us to identify the object $\Sigma z_3$ as isomorphic to $\Sigma z_2$.
From the second, we deduced by \autoref{thm:pasting} that the square
\begin{equation}
  \vcenter{\xymatrix{
      0_9\ar[r]\ar[d] &
      \Sigma z_2\ar[d]\\
      0_{11}\ar[r] &
      \Sigma z_4
      }}
\end{equation}
is also cocartesian, allowing us to also identify $\Sigma z_4$ as isomorphic to $\Sigma z_2$.
Now the commutativity of the squares
\begin{equation}
  \vcenter{\xymatrix{
      \Sigma y\ar[r]\ar[d] &
      \Sigma y\oplus \Sigma z\ar[d]\\
      0_{11}\ar[r] &
      \Sigma z_4
      }}
    \qquad\text{and}\qquad
  \vcenter{\xymatrix{
      \Sigma z_2\ar[r]\ar[d] &
      \Sigma y\oplus \Sigma z\ar[d]\\
      \Sigma z_3\ar[r] &
      \Sigma z_4
      }}
\end{equation}
implies that the composites $\Sigma y \to \Sigma y \oplus \Sigma z \to \Sigma z_4$ and $\Sigma z_2 \to \Sigma y \oplus \Sigma z \to \Sigma z_4$ are zero and the identity, respectively.
Therefore, the map $\Sigma y \oplus \Sigma z \to \Sigma z_4$ is in fact the projection out of $\Sigma y \oplus \Sigma z$ regarded as a product.
(It is also possible to construct \autoref{fig:extended-hypercube} using a cofiber sequence in $\shift\D{\Box\times\bbtwo}$.)

We now identify the composite $\Sigma x \to \Sigma y\oplus \Sigma z \to \Sigma z_4$ appearing in \autoref{fig:extended-hypercube} with $\Sigma g$.
By definition of the functor $\Sigma$, the map $\Sigma g$ is uniquely determined by occurring as the $(1,1)$-component of a morphism in $\D(\Box)$ from
\begin{equation}
  \vcenter{\xymatrix{
      x_2\ar[r]\ar[d] &
      0_4\ar[d]\\
      0_5\ar[r] &
      \Sigma x
    }}
  \qquad\text{to}\qquad
  \vcenter{\xymatrix{
      z_1 \ar[r]\ar[d] &
      0_7\ar[d]\\
      0_1\ar[r] &
      \Sigma z_2
      }}
\end{equation}
whose $(0,0)$-component is $g\colon x\to z$.
An obvious way to obtain such a morphism would be if we could find a coherent cube
\begin{equation}
  \vcenter{\xymatrix@C=.7pc@R=1pc{
    & x_1 \ar[dl] \ar[rrr] \ar'[d][dd]
    &&& 0_4 \ar[dl] \ar[dd] \\
    0_5 \ar[dd] \ar[rrr]
    &&& \Sigma x \ar[dd]\\
    & z_1 \ar[dl] \ar'[rr][rrr]
    &&& 0_7 \ar[dl] \\
    0_1 \ar[rrr]
    &&& \Sigma z_2.
  }}
\end{equation}
Unfortunately, there is no map from $0_4$ to $0_7$ in \autoref{fig:extended-hypercube}, so we cannot obtain such a cube by restriction.
However, we can instead obtain a ``coherent zigzag'' of cubes, as shown in \autoref{fig:sigma-g}.
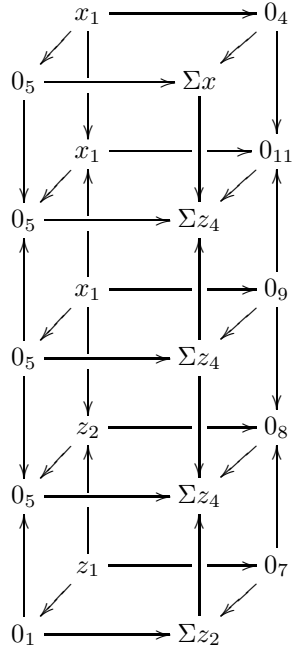
\begin{figure}
  \centering
  \begin{equation}\label{eq:Sigma-g}
  \vcenter{\xymatrix@C=.7pc@R=1pc{
    & x_1 \ar[dl] \ar[rrr] \ar'[d][dd]
    &&& 0_4 \ar[dl] \ar[dd] \\
    0_5 \ar[dd] \ar[rrr]
    &&& \Sigma x \ar[dd]\\
    & x_1 \ar[dl] \ar'[rr][rrr] \ar@{<-}'[d][dd]
    &&& 0_{11} \ar[dl] \ar@{<-}[dd] \\
    0_5 \ar@{<-}[dd] \ar[rrr]
    &&& \Sigma z_4 \ar@{<-}[dd]\\
    & x_1 \ar[dl] \ar'[rr][rrr] \ar'[d][dd]
    &&& 0_{9} \ar[dl] \ar[dd] \\
    0_5 \ar[dd] \ar[rrr]
    &&& \Sigma z_4 \ar[dd]\\
    & z_2 \ar[dl] \ar'[rr][rrr] \ar@{<-}'[d][dd]
    &&& 0_8 \ar[dl] \ar@{<-}[dd] \\
    0_5 \ar@{<-}[dd] \ar[rrr]
    &&& \Sigma z_4 \ar@{<-}[dd]\\
    & z_1 \ar[dl] \ar'[rr][rrr]
    &&& 0_7 \ar[dl] \\
    0_1 \ar[rrr]
    &&& \Sigma z_2
  }}
\end{equation}
\caption{Identification of $\Sigma g$}
\label{fig:sigma-g}
\end{figure}
This is an object of $\D(\Box\times F)$, where $F$ is the diagram shape
\[(\cdot \to \cdot \leftarrow \cdot \to \cdot \leftarrow \cdot).\]

We then apply the partial underlying diagram functor to \autoref{fig:sigma-g}, obtaining an incoherent $F$-shaped diagram in $\D(\Box)$.
However, since all the upwards-pointing arrows in \autoref{fig:sigma-g} are isomorphisms, by (Der2) so are the corresponding morphisms in $\D(\Box)$.
Thus, we can compose with their inverses to obtain a composite morphism
\[
\left(\vcenter{\xymatrix@R=1pc@C=.5pc{x_1 \ar[r] \ar[d] & 0_4 \ar[d] \\ 0_5 \ar[r] & \Sigma x}}\right)
\to
\left(\vcenter{\xymatrix@R=1pc@C=.5pc{x_1 \ar[r] \ar[d] & 0_{11} \ar[d] \\ 0_5 \ar[r] & \Sigma z_4}}\right)
\xleftarrow{\cong}
\left(\vcenter{\xymatrix@R=1pc@C=.5pc{x_1 \ar[r] \ar[d] & 0_6 \ar[d] \\ 0_5 \ar[r] & \Sigma z_4}}\right)
\to
\left(\vcenter{\xymatrix@R=1pc@C=.5pc{z_2 \ar[r] \ar[d] & 0_8 \ar[d] \\ 0_5 \ar[r] & \Sigma z_4}}\right)
\xleftarrow{\cong}
\left(\vcenter{\xymatrix@R=1pc@C=.5pc{z_1 \ar[r] \ar[d] & 0_7 \ar[d] \\ 0_1 \ar[r] & \Sigma z_2}}\right)
\]
in $\D(\Box)$.
Since the domain and codomain of this morphism are cocartesian squares, and its $(0,0)$-component is $g$, its $(1,1)$-component must be $\Sigma g$.

A symmetrical argument implies we can identify the map $\Sigma x \to \Sigma y$ with $\Sigma f$.
However, in this case the domain of the corresponding morphism in $\D(\Box)$ will be
\[\vcenter{\xymatrix{
    x_2\ar[r]\ar[d] &
    0_5\ar[d]\\
    0_4\ar[r] &
    \Sigma x
  }}
\]
which is transposed relative to our above choice of~\eqref{eq:sigma-x} to identify $\Sigma x$.
Thus, when we make the identifications consistently, we obtain $-\Sigma f$ in the second case.
\end{proof}

\begin{rmk}
  Inspecting the proof of \autoref{thm:triang} in~\cite[Theorem~4.16]{groth:ptstab}, we see that \autoref{prop:pushout-dt} and \autoref{lem:detection} imply that the triangulation of a stable derivator is always \emph{strong} in the sense of~\cite[Definition~3.8]{may:traces}.
  (This use of ``strong'' is unrelated to the notion of a derivator being strong.)
\end{rmk}

\section{On the definition of stable derivators}
\label{sec:characterization}

Finally, we use \autoref{prop:pushout-dt} to show that all three notions of stability for a derivator are equivalent.

\begin{thm}\label{thm:stable}
  For a pointed derivator \D, the following are equivalent.
  \begin{enumerate}
  \item \D is stable.\label{item:ss1}
  \item \D is cofiber-stable.\label{item:ss2}
  \item The adjunction $\cof\dashv\fib$ is an equivalence.\label{item:ss3}
  \item \D is $\Sigma$-stable.\label{item:ss4}
  \item The adjunction $\Sigma\dashv\Omega$ is an equivalence.\label{item:ss5}
  \end{enumerate}
\end{thm}
\begin{proof}
  Clearly~\ref{item:ss1}$\Rightarrow$\ref{item:ss2}$\Rightarrow$\ref{item:ss4}.
 By the construction of the adjunctions $\Sigma\dashv\Omega$ and $\cof\dashv\fib$ in Lemmas~\ref{lem:susploopadj} and~\ref{lem:coffibadj}, we easily deduce~\ref{item:ss2}$\Leftrightarrow$\ref{item:ss3} and~\ref{item:ss4}$\Leftrightarrow$\ref{item:ss5}.

  On the other hand, assuming~\ref{item:ss5}, the suspension functor of \D is an equivalence, and therefore (using \autoref{lem:shiftedkan}) the suspension functor of $\D^\bbtwo$
is an equivalence.
  Since $\cof^3=\Sigma$, this implies that $\cof$ is also an equivalence (e.g.\ by using the ``two-out-of-six property'' for equivalences); hence~\ref{item:ss3} holds.
  (This argument can be found in~\cite{heller:stable}, among other places.)

  It remains to show~\ref{item:ss3}$\Rightarrow$\ref{item:ss1}, and here we can mostly mimic the proof of \cite[1.1.3.4]{lurie:ha}.
  Let $X\in\D(\ulcorner)$ be of the form $z \ot x \to y$; we want to show that the cocartesian square $(i_\ulcorner)_!X$ is also cartesian.
  (The dual argument will be identical.)
  Now $X$ can be left extended by zero to a diagram of the following form
  \begin{equation}
    \vcenter{\xymatrix@-1pc{ &0 \ar[d] \\ &x \ar[dl] \ar[dr] \\ z && y. }}\label{eq:Y}
  \end{equation}
  Let $B$ denote the shape of~\eqref{eq:Y}, and let $A$ be the category $(\cdot \ot\cdot\to\cdot\ot\cdot\to\cdot)$.
  Then there is a functor $r\colon  A^\rhd \times \mathord\ulcorner \to B$ such that if $Y$ is~\eqref{eq:Y}, then $r^*Y$ has the form shown in \autoref{fig:rstarY}.
  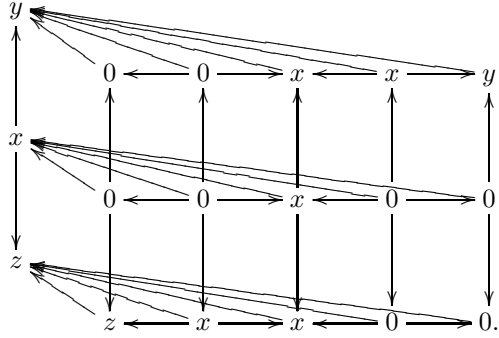
\begin{figure}
    \centering
  \[\vcenter{\xymatrix@R=1pc{
      y \ar@{<-}[dr] \ar@{<-}[drr] \ar@{<-}[drrr] \ar@{<-}[drrrr] \ar@{<-}[drrrrr] \ar@{<-}[dd] \\
      &0\ar@{<-}[r]\ar@{<-}[dd] &
      0\ar[r]\ar@{<-}[dd] &
      x\ar@{<-}[r]\ar@{<-}[dd] &
      x\ar[r]\ar@{<-}[dd] &
      y\ar@{<-}[dd]\\
      x \ar@{<-}[dr] \ar@{<-}[drr] \ar@{<-}[drrr] \ar@{<-}[drrrr] \ar@{<-}[drrrrr] \ar[dd]\\
      & 0\ar@{<-}[r]\ar[dd] &
      0\ar[r]\ar[dd] &
      x\ar@{<-}[r]\ar[dd] &
      0\ar[r]\ar[dd] &
      0\ar[dd]\\
      z \ar@{<-}[dr] \ar@{<-}[drr] \ar@{<-}[drrr] \ar@{<-}[drrrr] \ar@{<-}[drrrrr] \\
      &z\ar@{<-}[r] &
      x\ar[r] &
      x\ar@{<-}[r] &
      0\ar[r] &
      0.
    }}
  \]
    \caption{Building a span as a colimit of simpler ones}
    \label{fig:rstarY}
  \end{figure}
  It is straightforward to conclude that each vertical level of this diagram is in the image of $(i_A)_!$, where $i_A\colon A\to A^\rhd$ is the inclusion.
  Therefore, by \autoref{lem:shiftedkan}, the whole diagram is in the image of $(i_A\times 1_{\ulcorner})_!$.
  It follows that  $(\pi_A)_! (i_A\times 1_{\ulcorner})^* r^* Y\cong X$, where $\pi_A$ denotes the projection $A\times \mathord\ulcorner \to \mathord\ulcorner $.
  That is, if we view $(i_A\times 1_{\ulcorner})^* r^*Y$ as an $A$-shaped diagram in $\shift\D\ulcorner$, then its colimit is $X$.

  The diagram 
  \[\xymatrix{ A\times \ulcorner\ar[r]^-{\pi_A}\ar[d]_{1_{A}\times i_\ulcorner}&\ulcorner\ar[d]^{i_\ulcorner}\\A\times \Box\ar[r]^-{\pi_A}&\Box}\]
  commutes, so the colimit of 
  $(1_{A}\times i_\ulcorner)_! (i_A\times 1_{\ulcorner})^* r^*Y \in \shift\D\Box(A)$ as an $A$-shaped diagram is  $(i_\ulcorner)_!X$.
  However, this diagram has the following form in \D
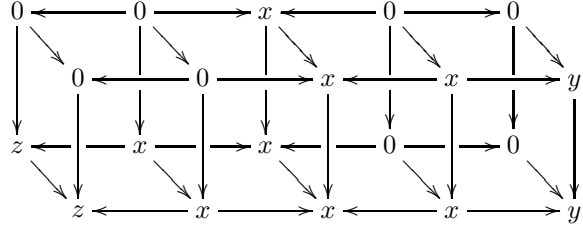
  \begin{equation}
    \vcenter{\xymatrix@-1pc{
        0 \ar@{<-}[rr] \ar[dr] \ar[dd]
        && 0 \ar[rr] \ar'[d][dd] \ar[dr]
        && x \ar@{<-}[rr] \ar[dr] \ar'[d][dd]
        && 0 \ar[rr] \ar'[d][dd] \ar[dr]
        && 0 \ar'[d][dd] \ar[dr] \\
        & 0 \ar@{<-}[rr] \ar[dd]
        && 0 \ar[rr] \ar[dd]
        && x \ar@{<-}[rr] \ar[dd]
        && x \ar[rr] \ar[dd]
        && y \ar[dd]\\
        z \ar@{<-}'[r][rr] \ar[dr]
        && x \ar'[r][rr] \ar[dr]
        && x \ar@{<-}'[r][rr] \ar[dr]
        && 0 \ar'[r][rr] \ar[dr]
        && 0 \ar[dr]\\
        & z \ar@{<-}[rr]
        && x \ar[rr]
        && x \ar@{<-}[rr]
        && x \ar[rr]
        && y
      }}
  \end{equation}
  and when regarded as an $A$-diagram in $\shift\D\Box$, all its objects are cartesian as well as cocartesian squares in \D, since they are constant in at least one direction (see \autoref{thm:trivial-pushout}).
  Therefore, it will suffice to show that cartesian squares are closed under $A$-shaped colimits in $\shift\D\Box$.

  Towards this end, we first note that $A$-shaped colimits can be constructed from pushouts.
  Namely, if $D$ denotes the following category
  \begin{equation}
    \vcenter
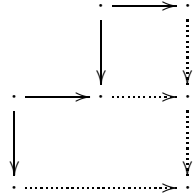
{\xymatrix{
        &\cdot \ar[r]\ar[d] &
        \cdot \ar@{.>}[d]\\
        \cdot \ar[r]\ar[d] &
        \cdot \ar@{.>}[r] &
        \cdot \ar@{.>}[d]\\
        \cdot \ar@{.>}[rr] & &
        \cdot
      }}
  \end{equation}
  with $j\colon A\into D$ the inclusion of the solid arrows, then in a diagram of the form $j_!Z$ both the square and the rectangle are cocartesian (by \autoref{lem:detection}) while the lower-right corner is the colimit over $A$ (by \autoref{lem:detectionplus} with $B'=C=A$).
  Thus, if cartesian squares are closed under pushouts, they are also closed under $A$-colimits.

  Now since $\cof\colon \shift\D{\Box\times\bbtwo} \to \shift\D{\Box\times\bbtwo}$ is an equivalence of derivators, it preserves cartesian squares; i.e.\ cartesian squares are closed under cofibers.
  In particular, they are closed under suspension (and under loop spaces).

  On the other hand, if $X$ and $Y$ are cartesian squares in \D, then the following squares in $\shift\D\Box$ (i.e.\ objects of $\shift\D\Box(\Box)$) are cocartesian
  \[\vcenter{\xymatrix@-.5pc{
      \Omega X\ar[r]\ar[d] &
      0\ar[d]\\
      0\ar[r] &
      X
    }}
  \qquad\text{and}\qquad
  \vcenter{\xymatrix@-.5pc{
      0\ar[r]\ar[d] &
      Y\ar[d]\\
      0\ar[r] &
      Y
    }}
  \]
  and hence so is their coproduct
  \[\vcenter{\xymatrix{
      \Omega X\ar[r]\ar[d] &
      Y\ar[d]\\
      0\ar[r] &
      X\sqcup Y.
    }}
  \]
  Thus, $X\sqcup Y$ is the cofiber of a map from $\Omega X$ to $Y$, both of which are cartesian; hence it is also cartesian.
  Thus, cartesian squares are closed under coproducts.

  Finally, for an arbitrary cocartesian square
  \[\vcenter{\xymatrix@-.5pc{
      X\ar[r]\ar[d] &
      Y\ar[d]\\
      Z\ar[r] &
      W
    }}
  \]
  in $\shift\D\Box$, \autoref{prop:pushout-dt} yields a cocartesian square
  \[\vcenter{\xymatrix@-.5pc{
      X\ar[r]\ar[d] &
      Y\sqcup Z\ar[d]\\
      0\ar[r] &
      W.
    }}
  \]
  Thus, if $X$, $Y$, $Z$, and hence also $Y\sqcup Z$ are cartesian, then $W$ is the cofiber of a map between cartesian squares and hence is also cartesian.
  Thus, cartesian squares are closed under pushouts in $\shift\D\Box$, as desired.
\end{proof}

\begin{rmk}
  \autoref{thm:stable} can be regarded as a converse to \autoref{thm:triang}: if the structure defined on a pointed derivator \D in \S\ref{sec:pointedstable} makes $\D(\bbone)$ triangulated, then in particular the suspension functor must be an equivalence; hence \D is stable.
\end{rmk}

\appendix
\section{The calculus of mates}
\label{sec:mates}

We briefly recall a very useful tool called the \emph{calculus of mates} for natural transformations.  More information can be found in e.g.~\cite{ks:r2cats} or~\cite[1.1]{ayoub:six}.

Suppose given a square of functors containing a natural transformation
\[\vcenter{\xymatrix{
    \cA\ar[r]^{f^*} \ar[d]_{h^*} \drtwocell\omit{\alpha} &
    \cB\ar[d]^{k^*}\\
    \cC\ar[r]_{g^*} &
    \cD.
  }}
\]
If the functors $f^*$ and $g^*$ have left adjoints $f_!$ and $g_!$ respectively, then $\alpha$ has a \textbf{mate} transformation $\alpha_!\colon g_! k^* \to h^* f_!$, defined to be the composite
\begin{equation}
  g_! k^* \xto{g_! k^* \eta} g_! k^* f^* f_! \xto{g_!\alpha f_!} g_! g^* h^* f_! \xto{\varepsilon h^* f_!} h^* f_!
\end{equation}
where $\eta$ and $\varepsilon$ denote the unit and counit of the adjunctions $f_! \dashv f^*$ and $g_! \dashv g^*$, respectively.
Similarly, if instead $k^*$ and $h^*$ have right adjoints $k_*$ and $h_*$ respectively, then $\alpha$ has another mate $\alpha_* \colon f^* h_* \to k_* g^*$
\begin{equation}
  f^* h_* \xto{\eta f^* h_*} k_* k^* f^* h_* \xto{k_* \alpha h_*} k_* g^* h^* h_* \xto{k_* g^* \varepsilon} k_* g^*.
\end{equation}
These operations are inverses, in that if we reorient $\alpha_!$ to look like $\alpha$
\[\vcenter{\xymatrix{
    \cB\ar[r]^{k^*} \ar[d]_{f_!} \drtwocell\omit{\alpha_!} &
    \cD\ar[d]^{g_!}\\
    \cC\ar[r]_{h^*} &
    \cA,
  }}
\]
then apply the second mate-construction to it (which we can do since $f_!$ and $g_!$ have right adjoints, namely $f^*$ and $g^*$), then we have $(\alpha_!)_* = \alpha$, and dually.
(This follows from the triangle identities for the adjunctions $f_! \dashv f^*$ and $g_! \dashv g^*$.)

On the other hand, if all four functors $f^*$, $g^*$, $h^*$, and $k^*$ have left adjoints $f_!$, $g_!$, $h_!$, and $k_!$ respectively, then we can apply the \emph{first} mate-construction to $\alpha_!$ to obtain $(\alpha_!)_! \colon h_! g_! \to f_! k_!$.
In this case, we can also regard $\alpha$ as a transformation
\[\vcenter{\xymatrix{
    \cA\ar@{=}[r]\ar[d]_{g^* h^*} \drtwocell\omit{\alpha} &
    \cA\ar[d]^{k^* f^*}\\
    \cD\ar@{=}[r] &
    \cD.
  }}
\]
Then since $g^* h^*$ and $k^* f^*$ have left adjoints $h_! g_!$ and $f_! k_!$ respectively, we can construct a mate $\alpha_!\colon h_! g_! \to f_! k_!$, and we have $\alpha_! = (\alpha_!)_!$.

The mate-construction is functorial with respect to horizontal and vertical pastings of squares.
This can be expressed formally as an isomorphism of double categories; see~\cite{ks:r2cats}.
However, it is not a functor in the ordinary sense, and in particular the mate of an isomorphism need not be an isomorphism.

It is true, however, that if $h^*$ and $k^*$ are identities (or even equivalences), then $\alpha$ is an isomorphism if and only if $\alpha_!$ is so, and dually.
In particular, if $f^*$ and $g^*$ have left adjoints and $h^*$ and $k^*$ have right adjoints, then $\alpha_!\colon g_! k^* \to h^* f_!$ is an isomorphism if and only if $\alpha_* \colon f^* h_* \to k_* g^*$ is so, using the above fact about ``iterated first mate-constructions''.
This is relevant to the definition of homotopy exact square, \autoref{defn:hoexact}.

\bibliographystyle{alpha}
\bibliography{derivbicat}

\end{document}